\documentclass[12pt]{amsart}

\usepackage{amsmath,amsthm,amssymb,amsfonts}
\usepackage[noadjust,nosort]{cite}

\usepackage{fullpage}
\usepackage{graphicx}
\usepackage[dvipsnames]{xcolor}

\usepackage[bookmarks=true,
    colorlinks=true,
    linkcolor=purple,
    citecolor=red,
    filecolor=orange,
    menucolor=purple
    urlcolor=green,
    breaklinks=true]{hyperref}

\usepackage[all]{xy}

\numberwithin{equation}{subsection}

\newtheorem{theorem}[equation]{Theorem}
\newtheorem{proposition}[equation]{Proposition}
\newtheorem{lemma}[equation]{Lemma}
\newtheorem{corollary}[equation]{Corollary}

\newtheorem*{carrier}{Carrier Lemma}

\theoremstyle{definition}
\newtheorem{definition}[equation]{Definition}
\newtheorem{example}[equation]{Example}
\newtheorem{construction}[equation]{Construction}

\newtheorem{remark}[equation]{Remark}

\newtheorem*{remark*}{Remark}

\newtheorem{convention}[equation]{Convention}

\renewcommand{\leq}{\leqslant}
\renewcommand{\geq}{\geqslant}
\renewcommand{\emptyset}{\varnothing}

\newcommand{\algebra}[1]{\mathbb{#1}}
\newcommand{\Rr}{\algebra{R}}

\newcommand{\setsystem}[1]{\mathcal{#1}}
\newcommand{\TT}{\setsystem{T}}
\newcommand{\UU}{\setsystem{U}}
\newcommand{\VV}{\setsystem{V}}
\newcommand{\power}{\setsystem{P}}

\newcommand{\id}{\mathrm{id}}
\newcommand{\proj}{\mathrm{proj}}
\newcommand{\conv}{\mathrm{conv}}

\newcommand{\nerve}{\mathrm{Nerve}}

\newcommand{\prodc}{\mathrm{prod}}
\newcommand{\simp}{\mathrm{simp}}

\newcommand{\order}{\mathrm{ord}}
\newcommand{\ocl}{\order_{\textsc{cl}}}

\newcommand{\closedsets}{\mathrm{Cl}}

\newcommand{\all}{\mathrm{all}}
\newcommand{\uall}{U^{\all}}
\newcommand{\vall}{V^{\all}}

\newcommand{\cone}{\mathrm{Cone}}
\newcommand{\link}{\mathrm{Link}}
\newcommand{\down}[1]{{\downarrow}{#1}}
\newcommand{\up}[1]{{\uparrow}{#1}}
\newcommand{\op}[1]{{#1}^{\mathrm{op}}}

\newcommand{\map}{\mathrm{Map}}

\newcommand{\fposet}{\mathbf{fPoset}}
\newcommand{\cl}{\mathbf{CL}}
\newcommand{\hcl}{\mathbf{hCL}}

\newcommand{\topcat}{\mathbf{Top}}
\newcommand{\htop}{\mathbf{hTop}}

\title{Topological Realizations of Finite Posets}

\author{Vin de Silva}
\address{Department of Mathematics \& Statistics, Pomona College, Claremont, California}

\begin{document}

\begin{abstract}
    I introduce a new class of cell complexes, called CL-complexes, that provide alternative topological realizations of finite posets. Each realization of a given poset has the same homotopy type as the order complex of that, so in some sense these spaces do not provide anything really new.
    Their value is that they enable short transparent proofs of standard results in poset homotopy. These include Quillen's Theorem~A for finite posets, the Nerve Theorem for good covers of a topological space, and Sheehy's multicover generalization of the Nerve Theorem.
    Order complexes retain a special status: the order complex of a finite poset admits CL-complex structures both for the poset itself and for its opposite poset.
    The definition of CL-complexes is motivated by an elementary generalization of the Carrier Lemma of J.~W.\ Walker, which in its original form constructs a canonical homotopy class of maps from a simplicial complex to a space covered by a suitable family of contractible sets.
    In the generalized lemma, the role played by simplicial complexes and the role played by spaces covered by contractible sets can both be fulfilled by CL-complexes.
    This sets up an equivalence of categories that tacitly underlies our approach.
\end{abstract}

\maketitle

\section{Introduction}
\label{sec:intro}

\subsection{Motivation}
\label{subsec:motivation}

The purpose of this paper is to introduce a new class of `cell complex' that provides alternative topological realizations of the homotopy type of a finite poset.
The standard---indeed, definitional---tool for converting posets into topological spaces is the order complex functor. Much of the established value of this functor resides at the homotopy level: we are usually interested in the homotopy type of the order complex, not its homeomorphism class.
The framework developed in this paper exploits this observation.
Each finite poset is assigned a large family of possible realizations that are homotopy equivalent to the order complex but not, in general, homeomorphic to it.
Each map of finite posets is assigned a canonical homotopy class of maps between respective realizations.

What is the advantage of doing this? 
My suggestion is that these new objects and their maps, which I call \emph{CL-complexes} and $\emph{CL-maps}$, provide transparent new ways of seeing known homotopy properties of finite posets. This plays out in various ways, some of which I enumerate here as proof strategies
(see Remark~\ref{rem:proof-strategies}):

\begin{itemize}
    \item[(S1)]
    Two posets are seen to be homotopy equivalent by exhibiting a single topological space that simultaneously realizes both posets.
    
    \medskip
    \item[(S2)]
    A map of posets is seen to be a homotopy equivalence by realizing it as a homeomorphism; most commonly as an identity map on a space that realizes both posets.
    
    \medskip
    \item[(S3)]
    Two poset maps are seen to be homotopic by exhibiting a map of topological spaces that simultaneously realizes both poset maps.

\end{itemize}
In the complementary direction:

\begin{itemize}
    \item[(S4)]
    Two topological spaces are seen to be homotopy equivalent by recognizing them as realizations of the same poset.

    \medskip
    \item[(S5)]
    A map between topological spaces is seen to be a homotopy equivalence by recognizing it as a realization of the identity map on a poset.

\end{itemize}
In these ways, the task of writing down homotopies and homotopy equivalences is replaced by the task of finding topological realizations that transparently reveal the result at hand.
A side effect is the absence of simplicial complexes from most of our reasoning. CL-complex arguments are desimplexified arguments.

The basis for this program is an adaptation of the Carrier Lemma, first articulated by Walker~\cite{Walker_1981}, following a partial statement by Lakser~\cite{Lakser_1971}, as a homotopy analogue of Eilenberg and Mac Lane's method of acyclic models in homological algebra~\cite{Eilenberg_MacLane_1953}.

\begin{carrier}[trad.\ version \cite{Walker_1981}]
    Let $K$ be a simplicial complex with geometric realization $|K|$. Let $X$ be a topological space with a `contractible carrier' for $K$, meaning a family of contractible subsets $(X^\sigma \mid \sigma \in K)$ indexed by the simplices of~$K$, which is monotone in the sense that $\sigma \leq \tau$ implies $X^\sigma \subseteq X^\tau$. Then:

    \quad
    (1) There exists a continuous map $f : |K| \to X$ with $f(|\sigma|) \subseteq X^{\sigma}$ for all $\sigma \in K$. 
    
    \quad
    (2) Any two such maps are homotopic through a family of other such maps.

    \noindent
    Thus there is a canonically specified homotopy class of maps $|K| \to X$ in this situation.
\end{carrier}

The statement is asymmetric: the domain is a geometric simplicial complex, the codomain is an arbitrary space with contractible carrier. In the present work, we obtain a symmetric version of the Carrier Lemma where the same type of object may be found on both sides of the map. These are the CL-complexes, which may be read as `Cellular Link-cofibrant complexes' (referring to their defining properties) or alernatively `Carrier Lemma complexes' (referring to how we use them).

We begin by weakening\footnote{%
Except for `finite' which is a stronger hypothesis than `not necessarily finite'.
} the hypotheses in such a way that the proof still works.
It is expedient to state this in the `relative' form concerning a map of posets.

\begin{carrier}[weakened hypotheses, Theorem~\ref{thm:carrier}]
    Let $\phi: P \to Q$ be a map of finite posets. Let $X,Y$ be topological spaces with monotone families of subsets $(X^p)$, $(Y^q)$ indexed by $P,Q$ respectively.
    Suppose the family $(X^p)$ is separated (Section~\ref{subsec:separated}) and link-cofibrant (Section~\ref{subsec:link-cofibrant}). Suppose that the subsets $Y^q$ are contractible. Then there is a unique homotopy class of continuous maps $f : X \to Y$ with the property that $f(X^p) \subseteq Y^{\phi(p)}$ for all~$p$.
\end{carrier}

Combining the hypotheses on the domain and codomain spaces,
we define
a CL-complex for a finite poset~$P$ to be a topological space $X$ with a separated, link-cofibrant monotone family $(X^p)$ of contractible subsets indexed by~$P$.

\begin{carrier}[CL-complex version, Theorem~\ref{thm:CL-map}]
    Let $\phi: P \to Q$ be a map of finite posets. Let $X,Y$ be CL-complexes for $P,Q$ respectively.
    Then there is a unique homotopy class of continuous maps $f : X \to Y$ with the property that $f(X^p) \subseteq Y^{\phi(p)}$ for all~$p$.  
\end{carrier}

Several consequences flow from this: most importantly, CL-complexes realizing the same finite poset are canonically homotopy equivalent.
They are easy to build: every finite poset admits CL-complex realizations, and the class of CL-complexes is closed under subcomplexes, finite products, and certain gluing operations.
Finite regular CW-complexes, including finite simplicial complexes, turn out to be CL-complexes for their own face posets; there is no need to go through barycentric subdivision.
The order complex retains a pivotal role: it admits CL-complex structures for both the poset itself and its opposite poset. This provides a useful bridge to certain applications such as the Nerve Theorem.

In one perspective, the benefits stem from an equivalence between the homotopy category of CL-complexes and the category of finite posets, where each finite poset is realized by CL-complexes of many different homeomorphism types that can be tailored to the application. It is this variety of representative spaces that enables some simple proofs (and one or two complicated proofs that may still provide some insight).
None of the applications is a genuinely new theorem. Rather, I suggest, they provide new heuristics for old results.

\subsection{Outline}
\label{subsec:outline}

The rest of the paper is organized as follows. Section~\ref{sec:posets} is about poset systems in a set. The key concept is a `separated' poset system, which I call a poset complex. Section~\ref{sec:topological-systems} concerns poset systems in a topological space. These may be closed, link-cofibrant, cellular. The Carrier Lemma is treated in Section~\ref{sec:carrier-lemma}. This paves the way to introduce CL-complexes in Section~\ref{sec:CL-complexes}. We slightly delay the construction of CL-complexes to a later subsection, first picking up some easy consequences and quick proofs. Sections \ref{sec:fiber-theorems} 
and~\ref{sec:nerve} give applications, to `fiber' theorems (such as Quillen's Theorem~A for finite posets, in absolute and relative forms) and theorems about nerves of covers, respectively. One highlight is an almost instant proof of Sheehy's multicover nerve theorem~\cite{Sheehy_2012}.

I offer a small caveat for the reader: the length of this paper may be somewhat misleading. The arguments needed to define, construct, and use CL-complexes are rather short, and can surely be condensed to two or three pages. This paper has become longer in order to give each concept its due, including examples, reformulations, and related tangents.

I leave as an open question the matter of handling infinite posets. Something should be possible for posets that are locally finite, in the sense that each element dominates at most finitely many other elements. The face posets for simplicial complexes and CW-complexes are of this type. Moreover, since every poset is homotopy equivalent to a locally finite poset (specifically its order complex viewed as a poset), this should give access to unrestricted versions of some of the results discussed here. On the other hand we would likely lose some of the slick proofs, which perhaps lessens the impetus for this work.

\section{Poset Systems in a Set}
\label{sec:posets}

In this section, we study poset systems in a set (without yet assuming a topology), and identify a `separation' condition that supports the inductive construction of functions. Eventually our posets will be finite, but we delay that requirement until it is needed.
Section~\ref{subsec:poset-terminology} sets up the terminology (some standard, some not). Sections \ref{subsec:P-systems} and~\ref{subsec:P-system-maps} introduce poset systems and maps between then. Section~\ref{subsec:separated} states the separation condition and derives various equivalent characterizations.

\subsection{Poset terminology}
\label{subsec:poset-terminology}

Let $P$ be a partially ordered set. We write $\op{P}$ for the `opposite poset', which has the same elements but the order relation reversed.
Unlike, say, the order complex of a poset, the concepts we introduce in this paper are not symmetric with respect to poset reversal.
We use language that is prejudiced in favour of the `down' properties that we need more often than the corresponding `up' properties.

Closed subsets of a poset play a central role in this paper.
A subset $A \subseteq P$ is called \emph{closed} (or \emph{down-closed}) if $a \in A$ and $p \leq a$ implies $p \in A$. We write $A \leq P$ to indicate that $A$ is a closed subset.
Given a subset $A \subseteq P$, the set
\[
\down{A} =
\{
p \in P \mid
\text{there exists $a \in A$ such that $p \leq a$}
\}
.
\]
is called the \emph{closure} (or \emph{down-closure}) of~$A$. It is the smallest closed subset of $P$ containing~$A$.
A closed set with a largest element is called a \emph{cone} (or \emph{down-cone}). We write
\[
\cone(p)
=
\down{\{p\}},
\quad
\link(p) = \cone(p) - \{p\}.
\]
These are called the \emph{cone} (or \emph{down-cone}) of~$p$ and the \emph{link} (or \emph{down-link}) of~$p$, respectively.

The terms \emph{up-closed}, \emph{up-closure}, \emph{up-cone}, \emph{up-link} and an operator $\up{}$ are defined similarly but in the opposite sense; here the prefix `up-' is compulsory and may not be omitted.

The {Alexandrov topology} on a poset \cite[\S1, paragraph 2]{Alexandroff_1937} specifies that a subset is topologically closed iff it is down-closed in the partial order. In this way, our chosen meanings of `closed' and `closure' coincide with their customary topological meanings.

A \emph{poset map} is a function between posets that is order-preserving in the sense that $p \leq p'$ implies $\phi(p) \leq \phi(p')$. If $\phi : P \to Q$ is a poset map, then $\op{\phi} : \op{P} \to \op{Q}$ will denote the same function but interpreted as a map between the opposites of the two posets.

The \emph{order complex} of a poset $P$, denoted $\order(P)$,\footnote{%
This is often written $\Delta(P)$ but in this paper I reserve $\Delta$ for complete simplicial complexes.
} is the set of nonempty finite chains in~$P$. It is a simplicial complex with geometric realization $|\order(P)|$. For a poset map $\phi : P \to Q$ we write $\order (\phi) : \order(P) \to \order(Q)$ for the induced map of simplicial complexes.
Note that $\order(P) = \order(\op{P})$ and $\order(\phi) = \order(\op{\phi})$.

\subsection{Poset systems}
\label{subsec:P-systems}

Let $P$ be a poset and let $X$ be a set. 
%
%
The constructions in this paper are based on the notion of a `poset system' (often called a `carrier' in the literature).

\begin{definition}[poset system]
A \emph{$P$-system in~$X$} is a family $(X^p)$ of subsets of~$X$, indexed by the elements of~$P$, which is monotone in the sense that $p \leq p'$ implies $X^p \subseteq X^{p'}$.
We extend the notation to write
\[
X^A = \bigcup_{a \in A} X^a
\]
for subsets $A \subseteq P$. This captures the notion of the region `covered' by~$A$.
Note that $X^A = X^{\down{A}}$ by monotonicity, and in particular $X^p = X^{\cone(p)}$.
We do not require $X^P = X$, meaning that the sets $(X^p)$ are not required to cover~$X$.
\end{definition}

\begin{definition}[indicator set]
Given a poset system $(X^p)$, we define the \emph{indicator set}
\[
P_x = \{ p \in P \mid x \in X^p \}
\]
of $x \in X$.
By monotonicity, this is an up-closed (or Alexandrov-open) subset of~$P$.
\end{definition}

\begin{remark}
\label{remar:notation-caution}
    The notation $X^p, P_x$ implicitly assumes that we are considering one fixed $P$-system in the set~$X$ at any given moment.
    This turns out to be sufficient for this paper, since we are not required to compare two or more $P$-systems in the same set.
    I leave it to the reader to devise suitable notation for other circumstances.
\end{remark}

\begin{example}[subsystem]
\label{ex:subsystem}
    Let $(X^p)$ be a $P$-system in~$X$ and let $A \subseteq P$.
    Then $(X^a) = (X^a \mid a \in A)$ is an $A$-system in~$X$. It is also an $A$-system in~$X^A$ (or any set containing $X^A$).
\end{example}

\begin{example}[tautological system]
    \label{ex:tautological-P-system}
    For any poset $P$, the formula
    \[
    P^p = \cone(p) = \down \{p\}
    \]
    defines a $P$-system in~$P$. We call this the \emph{tautological $P$-system}.
\end{example}

\begin{example}[simplicial complex]
    \label{ex:simplicial-P-system}
    Let $K$ be a simplicial complex and let $|K|$ denote its geometric realization.
    If we regard $K$ as a poset with respect to face inclusion, 
    then
    \[
    |K|^{\sigma} = |\sigma|
    \]
    defines a $K$-system of closed subsets in $|K|$: we associate to each simplex its corresponding geometric closed cell. We call this the \emph{face system} in~$|K|$.
\end{example}

\begin{remark}
\label{rem:K-tautological-face}
    The tautological system $(K^\sigma)$ on the simplicial complex $K$ viewed as a poset, and the face system $(|K|^\sigma)$ on the geometric realization $|K|$, are related by $|K|^\sigma = |K^\sigma|$.
\end{remark}

\begin{example}[standard $n$-simplex]
\label{ex:standard-n-simplex}
    As a special case of Example~\ref{ex:simplicial-P-system}, let $\Delta_n$ denote the set of non-empty subsets of $\{0, 1, \dots, n\}$. Its geometric realization $|\Delta_n|$ is the standard topological $n$-simplex. We write $\partial \Delta_n$ for $\Delta_n$ with its top element removed.
\end{example}

\begin{example}[CW-complex]
    \label{ex:CW-P-system}
    Let $X$ be a CW-complex, with open cells $(e_\alpha)$ indexed by a set $A$. Write $\overline{e_\alpha}$ for the closure of $e_\alpha$; this is the image of a closed disk $D_{\alpha}$ under a continuous map~$\Phi_{\alpha}$ called the `characteristic map' \cite[Appendix]{Hatcher_2002}.
    The transitive closure of the relation
    \[
    \overline{e_\alpha} \cap e_\beta \ne \emptyset
    \quad \Rightarrow \quad
    \beta \leq \alpha
    \]
    defines a poset structure on~$A$.
    Then
    \[
    X^{\alpha}
    = \bigcup_{\beta \leq \alpha} e_\beta
    = \bigcup_{\beta \leq \alpha} \overline{e_\beta}
    \]
    defines an $A$-system of subcomplexes of~$X$.
    We call this the \emph{face system} in~$X$, generalizing Example~\ref{ex:simplicial-P-system}.
    If $X$ is a normal CW-complex, meaning that every cell-closure is a subcomplex, then
    \[
    X^\alpha = \overline{e_\alpha}
    \]
    is a simpler equivalent definition.
    If $X$ is a regular CW-complex, meaning that $\Phi_\alpha : D_{\alpha} \to \overline{e_\alpha}$ is a homeomorphism for all~$\alpha$, then it is a normal CW-complex \cite[Chapter~III, Theorem~2.1]{Lundell_Weingram_1969} so the simpler definition applies in this case.
    
\end{example}

\begin{example}[product]
\label{ex:product}
    Let $(X^p)$ be a $P$-system in~$X$. Let $(Y^q)$ be a $Q$-system in~$Y$. Then
    \[
    (X \times Y)^{(p,q)}
    =
    X^p \times Y^q
    \]
    defines a $(P \times Q)$-system in $X \times Y$.
    %
\end{example}

\subsection{Maps of poset systems}
\label{subsec:P-system-maps}

Let $P, Q$ be posets. Let $(X^p), (Y^q)$ be poset systems in $X,Y$. Let $\phi : P \to Q$ be a poset map.
We define  \emph{$\phi$-map} to be a function $f : X \to Y$ such that
$
f(X^p) \subseteq Y^{\phi(p)}
$
for all $p \in P$.
When both posets are $P$, an $\id_P$-map is also called a \emph{$P$-map}.
As \emph{isomorphism} of $P$-systems $(X^p), (Y^p)$ is a bijection $f : X \to Y$ such that $f, f^{-1}$ are both $P$-maps; or, equivalently, such that $f(X^p) = Y^p$ for all $p$.

\begin{example}
\label{ex:poset-system-maps}
Here are some immediate examples.

(1)
    A poset map $\phi : P \to Q$ is a $\phi$-map for the tautological poset systems.
    
(2)
    Let $K,L$ be simplicial complexes on vertex sets $I,J$. A simplical map $K \to L$, defined as a map $\alpha : I \to J$ between vertex sets, induces a poset map $\phi_\alpha : K \to L$ by setting $\phi_\alpha(\sigma) = \alpha(\sigma)$, the image of the set~$\sigma$ under $\alpha$.
    The geometric realization $|\alpha| : |K| \to |L|$ is a $\phi_\alpha$-map with respect to the face systems on $|K|, |L|$. 

(3)
    Let $(X^p)$, $(Y^q)$ be $P, Q$ systems. Let $(X^p \times Y^q)$ be the product $(P\times Q)$-system.
    Let
    \begin{align*}
        \pi_P : & \xymatrix{P \times Q \ar[r] & P}
        &
        \pi_Q : & \xymatrix{P \times Q \ar[r] & Q}
        \\
        \pi_X : & \xymatrix{X \times Y \ar[r] & X}
        &
        \pi_Y : & \xymatrix{X \times Y \ar[r] & Y}
    \end{align*}
    denote the projection maps.
    Then $\pi_X$ is a $\pi_P$-map and $\pi_Y$ is a $\pi_Q$-map.
    This product system has the universal property of being the {largest} $(P\times Q)$-system on $X \times Y$ for which $\pi_X$ is a $\pi_P$-map and $\pi_Y$ is a $\pi_Q$-map.
    %

(4)
    Let $P,Q,R$ be posets. Let $(X^p), (Y^q), (Z^r)$ be poset systems on $X,Y,Z$. Suppose $\phi : P \to Q$ and $\psi : Q \to R$ are poset maps. If $f : X \to Y$ is a $\phi$-map and $g : Y \to Z$ is a $\psi$-map then $gf$ is a $\psi\phi$-map.
\end{example}

\begin{example}[$f$-pullback]
\label{ex:f-pullback}
    Let $(X^p)$ be a $P$-system in~$X$. Let $f : W \to X$ be a function. Then
    \[
    W^p = f^{-1}(X^p)
    \]
    defines a $P$-system on~$W$.
    It is the largest $P$-system on $W$ for which $f$ is a $P$-map.
\end{example}

\begin{example}[$\phi$-pushforward]
\label{ex:phi-pushforward}
    Let $(X^p)$ be a $P$-system in~$X$. Let $\phi : P \to Q$ be a poset map. Then
    \[
    X^q =
    \bigcup \,
    \{ X^p \mid p \in P, \, \phi(p) \leq q \}
    \]
    defines a $Q$-system in~$X$. It is the smallest $Q$-system for which $\id_X$ is a $\phi$-map.
\end{example}

\begin{remark}[Quillen fiber \cite{Quillen_1973}]
\label{rem:Quillen-fiber}
    The indexing set for the union in Example~\ref{ex:phi-pushforward} is called a \emph{Quillen (down-)fiber} for the map~$\phi$. It is a closed subset of~$P$, sometimes written:
    \[
    \phi/q
    = \phi^{-1}(Q_{\leq q})
    = \{ p \in P \mid \phi(p) \leq q \}
    \]
    Thus we may define the $\phi$-pushforward by $X^q = X^{\phi/q}$.
    Similarly, the \emph{Quillen up-fiber} is
    \[
    q \backslash \phi
    =
    \phi^{-1}(Q_{\geq q})
    =
    \{ p \in P \mid q \leq \phi(p) \}.
    \]
    This is up-closed, and $(X^{q \backslash \phi})$ is a $\op{Q}$-system rather than a $Q$-system.
\end{remark}

The observation, from Example~\ref{ex:phi-pushforward}, that $\id_X : X \to X$ is a $\phi$-map with respect to the $P$-system $(X^p)$ and the $\phi$-pushforward $Q$-system $(X^{\phi/q})$, will be used later on (Theorem~\ref{thm:QTA}, Example~\ref{ex:product-fiber-CL}). Here is a more general statement (for use in Theorem~\ref{thm:QTA-rel}).

\begin{proposition}
\label{prop:push-square}
    Suppose we have a commutative square of poset maps:
    \[
    \xymatrix{
    P
        \ar[r]^-{\phi}
        \ar[d]_-{\alpha}
    &
    Q
        \ar[d]^-{\beta}
    \\
    R
        \ar[r]^-{\psi}
    &
    S
    }
    \]
     Let $f : X \to Y$ be a $\phi$-map with respect to poset systems $(X^p)$, $(Y^q)$. Then:
     
    (1) $\phi$ restricts to a map $\alpha/r \to \beta/\psi(r)$ between Quillen fibers, for all $r \in R$.
        
    (2) $f$ is a $\psi$-map with respect to the pushforward poset systems $(X^{\alpha/r})$, $(Y^{\beta/s})$.
\end{proposition}

\begin{proof}
    For (1), given $p \in \alpha/r$, meaning that $\alpha(p) \leq r$, we have $\beta(\phi(p)) = \psi(\alpha(p)) \leq \psi(r)$ so $\phi(p) \in \beta/\psi(r)$.
    Statement (2) follows from (1) and the hypothesis that $f$ is a $\phi$-map.
\end{proof}

\subsection{Separated poset systems}
\label{subsec:separated}

To define a function iteratively over a set $X$ using a poset system, we ask that the sets $X^p$ do not `interfere' with each other beyond what is implied by the poset structure. This is the essence of the separation condition, although the initial definition does not directly state this.

\begin{definition}[separated poset system]
\label{def:separated}
    A $P$-system $(X^p)$ is \emph{separated}\footnote{%
    The choice of the word `separated' is motivated by the observation in Example~\ref{ex:minimal-disjoint}.} %
    if each indicator set $P_x$ has a smallest element. Writing $p_x$ for the smallest element, it follows that $P_x = \up{\{p_x\}}$.
\end{definition}

\begin{definition}[poset complex]
\label{def:P-complex}
    A \emph{$P$-complex} is a set $X$ with a separated $P$-system $(X^p)$.
\end{definition}

For a separated system we have $X^P = X$, meaning that the system $(X^p)$ covers~$X$. Indeed each $P_x$ is non-empty so every $x \in X$ is contained in some $X^p$.

\begin{convention}
\label{conv:XP}
    We will frequently use the symbol $X^P$ as shorthand for a poset complex specified by a set~$X$ and a separated $P$-system $(X^p)$. 
\end{convention}

This is ambiguous in that $X^P$ refers also to the subset covered by~$P$.
Since $X^P = X$ for a separated system, the ambiguity is mostly harmless: the poset complex~$X^P$ is the set $X^P$ with additional structure.\footnote{%
By way of comparison, a topological space is usually called $X$ rather than $(X,\tau)$.}
The intended meaning should be clear from the context.

Here are some examples of poset complexes.

\begin{example}
\label{ex:separation-behaves}
The following poset systems are separated:

    (1) The tautological system on a poset.

    (2) The face system on a simplicial complex or CW-complex.

    (3) The product of two separated systems.

    (4) The $f$-pullback of a separated system.

    (5) The $\phi$-pushforward of a separated system.

\end{example}

\begin{proof}
    (1) We have $P_p = \up{\{p\}}$ with smallest element $p$.

    (2) Let $e_\alpha$ be the unique open cell containing $x \in X$. Then $x \in X^\beta$ implies that $x \in \overline{e_{\gamma}}$ for some $\gamma \leq \beta$. Thus $e_{\alpha} \cap \overline{e_{\gamma}}$ is nonempty, so $\alpha \leq \gamma \leq \beta$. It follows that $\alpha$ is the smallest element of $A_x$.


    (3) We have $(P \times Q)_{(x,y)} = P_x \times Q_y$ with smallest element $(p_x,q_y)$.

    (4) From $X^p = f^{-1}(Y^p)$ we get $P_x = P_{f(x)}$ with smallest element $p_{f(x)}$.

    (5) From $X^q = X^{\phi/q}$ we get $Q_x = \up{\,\phi(P_x)}$ with smallest element $\phi(p_x)$.
    
\end{proof}

\begin{remark}
\label{rem:product-notation}
    The product of $X^P$ and $Y^Q$ may be written $X^P \times Y^Q$ or $(X \times Y)^{P \times Q}$.
\end{remark}

The next two examples begin to reveal the structure of the separation condition.

\begin{example}[minimal elements]
\label{ex:minimal-disjoint}
    Let $X^P$ be a $P$-complex. For \emph{minimal} elements $a \in P$, the sets $X^a$ are mutually disjoint. To see this, note that for any $x \in X$ there is at most one element of the indicator set $P_x$ that can be minimal in~$P$, namely the smallest element~$p_x$.
    This observation motivates the word `separated'.
\end{example}

\begin{example}[$V$-complex]
\label{ex:V-complex}
    Let $V = (1 > 0 < 2)$ denote the V-shaped 3-element poset with a smallest element~0 and two maximal elements $1,2$.
    A $V$-system in a set~$X$ is specified by subsets $X^0, X^1, X^2 \subseteq X$ with $X^0 \subseteq X^1$ and $X^0 \subseteq X^2$. If it is separated then
    \begin{align*}
        X^0 &= \{ x \in X \mid p_x = 0 \} 
        \\
        X^1 &= \{ x \in X \mid p_x = 0 \text{\;or\;} 1 \} 
        \\
        X^2 &= \{ x \in X \mid p_x = 0 \text{\;or\;} 2 \} 
    \end{align*}
    from which it follows that $X = X^1 \cup X^2$ and $X^0 = X^1 \cap X^2$.
    In this way, we see that a $V$-complex is the `same thing' as a set $X$ with a cover by two subsets $X^1, X^2$.
\end{example}

\begin{example}
\label{ex:subcomplex}
     Let $X^P$ be a $P$-complex and let $A \leq P$. Then $X^A$ is an $A$-complex.
\end{example}

Here $X^A$ refers to the set $X^A$ with the poset system $(X^a \mid a \in A)$.

\begin{proof}
    Let $x \in X^A$. This means that $A_x = P_x \cap A$ is nonempty. Since $A$ is closed, $A_x$ must also contain the smallest element $p_x$ of $P_x$, which is therefore the smallest element of~$A_x$.
\end{proof}

We call $X^A$ a \emph{subcomplex} of $X^P$ in the scenario above.
In general, a sub-{\it system}
of a separated system need not be separated. For example, let $X^V$ be a $V$-complex
(Example~\ref{ex:V-complex})
with $X^0$ nonempty and let $A = \{1,2\}$. Then $(X^a)$ is not a separated $A$-system because $X^1, X^2$ are not disjoint.
Here $A \subseteq V$ but $A \not\leq V$.
%

There is a certain rigidity to separated systems.
Let $\closedsets(P)$ denote the set of closed subsets of the poset~$P$. Let $\power(X)$ denote the powerset of the set~$X$. These are bounded lattices: partially ordered by inclusion, their algebraic operations include pairwise meet ($\cap$), pairwise join ($\cup$), smallest element ($\emptyset$), largest element ($P, X$ respectively).

\begin{proposition}[bounded lattices interpretation]
\label{prop:bounded-lattice-map}
    Let $(X^p)$ be a $P$-system in a set~$X$. Consider the following properties:

    \begin{enumerate}
        \item
        $(X^p)$ is separated.

        \medskip
        \item
        The mapping $A \mapsto X^A$ defines a bounded-lattice homomorphism 
        $\closedsets(P) \to \power(X)$.
        
    \end{enumerate}
     Then (1) $\Rightarrow$ (2). If $P$ is finite (or, more generally, well founded\,\footnote{%
        A poset is well founded if every nonempty subset has a minimal element.
        }) then (2) $\Rightarrow$ (1).
\end{proposition}

Here `bounded-lattice homomorphism' means a map between bounded lattices that preserves meet, join, smallest element, largest element.
For (2) this amounts to
\[
X^{A \cap B} = X^A \cap X^B,
\quad
X^{A \cup B} = X^A \cup X^B,
\quad
X^\emptyset = \emptyset,
\quad
X^P = X
\]
whenever $A, B \leq P$.

\begin{proof}
    Assume (1).
    Even without (1) we have
    \[
    X^{\emptyset} = \emptyset,
    \quad
    X^{A \cup B} = X^A \cup X^B,
    \quad
    X^{A \cap B} \subseteq X^A \cap X^B
    \]
    for arbitrary $P$-systems and subsets $A, B \subseteq P$.
    With (1) we have $X^P = X$ since each $P_x$ is nonempty. 
    Now let $A, B \leq P$ and consider $x \in X^A \cap X^B$.
    Then $x \in X^a \cap X^b$ for some $a \in A$, $b \in B$. Let $p$ be the smallest element of $P_x$. Since $a,b \in X_p$ we have $p \leq a$ and $p \leq b$ and therefore $p \in A$ and $p \in B$. Thus $x \in X^p \subseteq X^{A \cap B}$. This argument establishes $X^A \cap X^B \subseteq X^{A \cap B}$ and therefore $X^{A \cap B} = X^A \cap X^B$.
    %

    Assume (2) and suppose $P$ is well founded.
    Since $X^P = X$, every $P_x$ is nonempty and therefore has a minimal element~$p$. We show that $p$ is the smallest element. Consider $q \in P_x$. Then
    \[
    x
    \in
    X^p \cap X^q
    =
    X^{\cone(p)} \cap X^{\cone(q)}
    =
    X^{\cone(p) \cap \cone(q)}
    \]
    so some element of $\cone(p) \cap \cone(q)$ belongs to $P_x$. By minimality, this element of $\cone(p)$ can only be~$p$. This implies $p \in \cone(q)$ so $p \leq q$.
    %
\end{proof}

Here is a very different characterization of separated systems.

\begin{proposition}[classifying map interpretation]
\label{prop:classifying-map}
    Let $X$ be a set, let $P$ be a poset, and let $P^P$ be the tautological $P$-complex (so $P^p = \cone(p)$ for all~$p$).

    \begin{enumerate}
    \item 
    There is a bijective correspondence between $P$-complex structures on~$X$ and functions $m : X \to P$. The correspondence is defined as follows.
 
    \begin{itemize}
    \medskip
    \item
        Given a $P$-complex $X^P$, define its \emph{classifying map} $m(x) = \min P_x$.
        
    \medskip
    \item
        Given a function $m : X \to P$, let $X^P$ be the $m$-pullback of $P^P$.

    \end{itemize}
    \end{enumerate}

    \noindent
    For such $X^P$ and $m : X \to P$, consider a map $f : W \to X$ and a poset map $\phi : P \to Q$.

    \begin{enumerate}
    \setcounter{enumi}{1}
    \item 
    The classifying map of the $f$-pullback $(W^p)$ of $(X^p)$ is $m \circ f$.

    \medskip
    \item
    The classifying map of the $\phi$-pushforward $(X^q)$ of $(X^p)$ is $\phi \circ m$.
    \end{enumerate}
 
\end{proposition}

\begin{proof}
    (1)
    We show that the two operations are mutually inverse.
    Let $X^P$ be a $P$-complex and define $m(x) = \min P_x$. Then $x \in X^p$ iff $p \in P_x$ iff $m(x) \leq p$. Thus $X^p = m^{-1}(\cone(p))$ so $X^P$ is the $m$-pullback of $P^P$.
    In the other direction,
    let $m : X \to P$ and define $X^P$ to be the $m$-pullback of $P^P$. Then $p \in P_x$ iff $x \in X^p$ iff $m(x) \in P^p$ iff $m(x) \leq p$, so $P_x = \up{\{m(x)\}}$ and $m$ is recovered by $m(x) = \min P_x$.
    
    (2), (3) follow from the identification of smallest elements in Example~\ref{ex:separation-behaves} (4), (5).
\end{proof}

\begin{remark}[cell structure]
The word `complex' alludes to the following `cell decomposition' of a poset complex~$X^P$.
\begin{alignat*}{3}
&
\text{`closed cell'}
&\qquad&
X^p
&&=
\left\{ x \in X \mid m(x) \leq p \right\}
\\
&
\text{`open cell'}
&&
X^p - X^{\link(p)}
&&= \left\{ x \in X \mid m(x) = p \right\}
\end{alignat*}
This becomes more meaningful when we consider topological poset complexes in Section~\ref{sec:topological-systems}.
Just as with CW-complexes, the `open' in `open cell' doesn't mean topologically open in~$X$.    
\end{remark}

\section{Poset Systems in a Topological Space}
\label{sec:topological-systems}

A \emph{topological poset system} is a poset system $(X^p)$ in a topological space $X$. If it is separated then we call $X^P$ a \emph{topological poset complex}. A {$\phi$-map} of topological poset systems is a continuous map that is a $\phi$-map of poset systems. An isomorphism of topological poset systems is a homeomorphism that is an isomorphism of poset systems. We may use phrases like `continuous $\phi$-map' or `topological isomorphism' so we don't forget.

In this section,
I describe three types of topological poset system---closed systems (Section~\ref{subsec:closed}), link-cofibrant complexes (Section~\ref{subsec:link-cofibrant}), cellular systems (Section~\ref{subsec:cellular})---and present a protocol for attaching topological poset systems to each other (Section~\ref{subsec:gluing}).
There are no assumptions on the underlying topological spaces. In particular, they need not be Hausdorff.

\subsection{Closed poset systems}
\label{subsec:closed}

A \emph{closed poset system} is a topological poset system $(X^p)$ with the property that $X^A$ is a closed subset of~$X$ for all $A \subseteq P$. This implies that all subsystem inclusions $X^A \subseteq X^B$ are closed maps.

It suffices to verify the condition for closed subsets $A \leq P$ because in general $X^A = X^{\down{A}}$. When $P$ is finite it suffices to verify that $X^p$ is closed for all $p \in P$, because then every $X^A$ is a finite union of closed sets.

For separated poset systems, we have the following characterization:

\begin{proposition}
\label{prop:cc-classifying}
    A topological $P$-complex $X^P$ is closed iff its classifying map $m : X \to P$  is continuous with respect to the Alexandrov topology on $P$.
\end{proposition}

\begin{proof}
    For $A \leq P$ we have $X^A = m^{-1}(\down{A}) = m^{-1}(A)$. Thus $X^P$ is closed iff the preimage $m^{-1}(A)$ is closed for every $A \leq P$, that is, for every Alexandrov-closed $A \subseteq P$.
\end{proof}

\begin{example}
\label{ex:closed}
The following topological poset systems are closed:

    (1) The tautological system on a poset with the Alexandrov topology.

    (2) The face system on a simplicial complex or CW-complex.

    (3) Any subsystem of a closed poset system.

    (4) The product of two closed poset complexes.

    (5) The $f$-pullback of a closed poset system, for a continuous map $f$.

    (6) The $\phi$-pushforward of a closed poset system, for a poset map $\phi : P \to Q$.
\end{example}

\begin{proof}
    The proofs of (1), (2), (3), (5), (6) are immediate. For (4), let $X^P, Y^Q$ be closed complexes with classifying maps $m : X \to P$, $n : Y \to Q$. Then $m \times n : X \times Y \to P \times Q$ is the classifying map for $X^P \times Y^Q$. Now $m,n$ are continuous by Proposition~\ref{prop:cc-classifying}, so $m \times n$ is continuous with respect to the product topology on $P \times Q$. This topology is the same as the (ostensibly finer) Alexandrov topology on $P \times Q$, because every Alexandrov-open set $U \subseteq P \times Q$ is a union of product open sets:
    \[
    U
    = \bigcup_{(p,q) \in U} \up{\{(p,q)\}}
    = \bigcup_{(p,q) \in U} \up{\{p\}} \times \up{\{q\}}
    \]
    Then $X^P \times Y^Q$ is closed by Proposition~\ref{prop:cc-classifying}.
\end{proof}

\subsection{Link-cofibrant complexes}
\label{subsec:link-cofibrant}

The following property enables a closed poset complex to play the role played by simplicial complexes in the traditional Carrier Lemma. Note that closedness is a required pre-condition.

\begin{definition}[link-cofibrancy]
A \emph{link-cofibrant} poset complex is a closed topological $P$-complex $X^P$ such that
for all $p \in P$ the inclusion
\[
\xymatrix{
X^{\link(p)}
    \ar[r]
&
X^{\cone(p)} = X^p
}
\]
is a cofibration. These inclusions are closed maps since $X^P$ is a closed complex.
\end{definition}

We review the properties of cofibrations. I recommend Brown \cite[Chapter 7]{Brown_2006} for a clear exposition.
A subspace inclusion map $U \to X$ is called a {cofibration}, or cofibrant inclusion, if it satisfies the {homotopy extension property}: given a continuous map $f : X \to Y$ to any target space~$Y$, any homotopy on~$U$ starting at $f|_U$ has an extension to a homotopy on~$X$ starting at~$f$.

\begin{proposition}[{\cite[7.3.2]{Brown_2006}}]
\label{prop:cf-composition}
    A composite of cofibrations is a cofibration.
    \qed
\end{proposition}

\begin{proposition}[special case of {\cite[7.3.4]{Brown_2006}}]
\label{prop:cf-MV}
Let $U, V$ be closed subspaces of a topological space. If the inclusion $U \cap V \to V$ is a cofibration then the inclusion $U \to U \cup V$ is a cofibration.
\qed
\end{proposition}

\begin{proposition}[subcomplex inclusions]
\label{prop:cf-subcomplexes}
    Let $X^P$ be a link-cofibrant complex over a finite poset. Then for any $A \leq B \leq P$ the inclusion $X^A \to X^B$ is a cofibration.    
\end{proposition}

\begin{proof}
    Consider the special case $B = A \cup \{p\}$ where $p \in P-A$ is minimal. The map
    \[
    \xymatrix{
    X^{\link(p)}
    =
    X^{A \cap \cone(p)} 
    =
    X^A \cap X^{\cone(p)}
    \ar[r]
    &
    X^{\cone(p)}
    }
    \]
    is a cofibration by hypothesis, so the map
    \[
    \xymatrix{
    X^A
    \ar[r]
    &
    X^A \cup X^{\cone(p)}
    =
    X^B
    }
    \]
    is a cofibration by Proposition~\ref{prop:cf-MV}. Now iterate this step using Proposition~\ref{prop:cf-composition}.
\end{proof}

The following property of closed cofibrations was proved by Str{\o}m~\cite{Strom_1968}.

\begin{proposition}[{\cite[Theorem~6]{Strom_1968}, \cite[7.3.8]{Brown_2006}}]
\label{prop:cf-product}
If the closed inclusions $U \to X$ and $V \to Y$ are cofibrations, then the closed inclusion $\big(U \times Y\big) \cup \big(X \times V\big) \to X \times Y$ is a cofibration.
\qed
\end{proposition}

\begin{example}
\label{ex:LC}
The following closed poset complexes are link-cofibrant:

(1) The face system on a simplicial complex or CW-complex.

(2) Any subcomplex of a link-cofibrant poset complex.

(3) The product of two link-cofibrant poset complexes.

(4) The $\phi$-pushforward of a link-cofibrant poset complex.
    
\end{example}

\begin{proof}
    These are closed complexes by Example~\ref{ex:closed}. We verify the cofibration property.

    (1) Subcomplex inclusions of simplicial complexes and CW-complexes are cofibrations \cite[Section~7.3, Example~5]{Brown_2006} so in particular each $X^{\link(\alpha)} \to X^{\cone(\alpha)} = X^{\alpha}$ is a cofibration.

    (2) Immediate.

    (3) Let $X^P, Y^Q$ be link-cofibrant complexes. For $(p,q) \in P \times Q$ the inclusions
    \[
    \xymatrix{
    X^{\link(p)} \ar[r] & X^{\cone(p)}
    }
    \quad
    \text{and}
    \quad
    \xymatrix{
    Y^{\link(q)} \ar[r] & Y^{\cone(q)}
    }
    \]
    are closed cofibrations by hypothesis.
    Now
    \begin{align*}
    \link((p,q))
    &=
    \link(p) \times \cone(q)
    \cup
    \cone(p) \times \link(q)
    \\
    \cone((p,q))
    &=
    \cone(p) \times \cone(q)
    \end{align*}
    therefore
    \begin{align*}
    (X \times Y)^{\link((p,q))}
    &=
    X^{\link(p)} \times Y^{\cone(q)}
    \cup
    X^{\cone(p)} \times Y^{\link(q)}
    \\
    (X \times Y)^{\cone((p,q))}
    &=
    X^{\cone(p)} \times Y^{\cone(q)}
    \end{align*}
    so the inclusion
    \[
    \xymatrix{
    (X \times Y)^{\link((p,q))}
        \ar[r]
    &
    (X \times Y)^{\cone((p,q))}
    }
    \]
    is a cofibration by Proposition~\ref{prop:cf-product}.

    (4) Let $X^Q$ be the $\phi$-pushforward of a link-cofibrant complex $X^P$. Each $X^{\link(q)} \to X^{\cone(q)}$ is an inclusion of subcomplexes of $X^P$ and hence a cofibration by Proposition~\ref{prop:cf-subcomplexes}.
\end{proof}

\begin{remark}[tautological complex]
    \label{rem:LC-tautological-no}
    The tautological complex on a poset is not, in general, link-cofibrant. Let $\Lambda = (1 < \infty > 2)$ be the $\Lambda$-shaped three-element poset with largest element $\infty$ and two minimal elements $1,2$. The inclusion of $\Lambda^{\link(\infty)} = \{1, 2\}$ into $\Lambda^{\infty} = \Lambda$ is not a cofibration: if $f$ is a constant map from $\Lambda$ to a Hausdorff space, then any homotopy starting at at $f|_{\{1,2\}}$ that separates the two points cannot be extended to $\Lambda$.
\end{remark}

\begin{remark}[pullback]
\label{rem:LC-pullback-no}
Link-cofibrancy is {not} preserved by $f$-pullback.
Consider the two-element ordered set $I = (0 < 1)$. 
Define $[0,1]^I$ by $[0,1]^0 = \{0\}$ and $[0,1]^1 = [0,1]$.
Let $C \subseteq [0,1]$ be any closed subset whose inclusion is not a cofibration, such as the Cantor set, and define $X^I$ by $X^0 = C$ and $X^1 = [0,1]$. Let $f : [0,1] \to [0,1]$ be the distance function $f(t) = d(t,C)$. Then $[0,1]^I$ is link-cofibrant, whereas its $f$-pullback is $X^I$ which is not link-cofibrant.
\end{remark}

We finish with a simple fact for later use.

\begin{proposition}
\label{prop:ccce}
    Every topological space admits a closed cofibrant inclusion into a contractible space, which may be taken to be compact Hausdorff if the original space is.
\end{proposition}

\begin{proof}
    The inclusion of a space into its cone is such a map \cite[Section 7.3, Example 3]{Brown_2006}.
\end{proof}

\subsection{Cellular systems}
\label{subsec:cellular}

Let $(X^p)$ be a poset system in a topological space $X$. We call it \emph{cellular} if the sets $X^p$ are contractible. This is not the usual use of the word, which implies Euclidean disks. Our cells can be terribly non-cell-like provided they are contractible. This suffices for our purposes because it suffices for the Extension Lemma~\ref{lem:extension}.

\begin{example}
\label{ex:cellular}
The following poset systems are cellular:

(1) The tautological system on a poset with the Alexandrov topology.

(2) The face system of a simplicial complex or regular CW-complex.

(3) Any subsystem of a cellular system.

(4) The product of two cellular systems.
\end{example}

\begin{proof}
    (1) Each cone is contractible in the Alexandrov topology. Indeed, the map
    \[
    H : \cone(p) \times [0,1] \longrightarrow \cone(p)
    \]
    defined by
    \[
    H(q,t) = 
    \begin{cases}
        q & \text{for $t=0$}
        \\
        p & \text{for $t>0$}
    \end{cases} 
    \]
    is continuous, because for any closed set $A \leq P$ the preimage
    \[
    H^{-1}(A) =
    \begin{cases}
            A \times \{0\}
            & \text{if $A \neq P$ (equivalently $p \not \in A$)}
            \\
            \cone(p) \times [0,1]
            & \text{if $A = P$}
    \end{cases}
    \]
    is closed. Thus $H$ is a homotopy from the identity map to the constant map~$p$.

    (2) Regular CW-complexes are normal \cite[Chapter~III, Theorem~2.1]{Lundell_Weingram_1969} so $X^\alpha = \overline{e_\alpha}$ for all $\alpha$. By regularity, each $\overline{e_\alpha}$ is a closed disk hence contractible.
    
    (3), (4) Immediate.
\end{proof}

Cellularity is {not} preserved by $f$-pullback or $\phi$-pushforward. 

\subsection{Attachment}
\label{subsec:gluing}

We can attach two poset complexes along a common subcomplex. We use this for the most general construction of CL-complexes (Theorem~\ref{thm:existence}).

\begin{proposition}
\label{prop:glued}
    Let $P$ be a poset, and suppose $P = A \cup B$ for closed sets $A, B \leq P$. Write $N = A \cap B$.
    Consider a poset system $(X^p)$ with $X^P = X$. Suppose $(X^a \mid a \in A)$ and $(X^b \mid b \in B)$ are separated and suppose $X^A \cap X^B = X^N$. Then $(X^p)$ is separated.
    
\end{proposition}

\begin{proof}
    Let $x \in X$. Since $X = X^P = X^A \cup X^B$ it follows that $P_x$ must meet at least one of $A,B$.
    If $P_x$ meets $A$ but not~$B$, then $P_x = P_x \cap A = A_x$ which has a smallest element. If $P_x$ meets $B$ but not~$A$, then $P_x = P_x \cap B = B_x$ which has a smallest element. Otherwise $P_x$ meets both $A$ and $B$, so $x \in X^A \cap X^B = X^N$ and therefore $P_x$ meets $N$.
    It follows that the smallest element of $P_x \cap A$ belongs to $N$, as does the smallest element of $P_x \cap B$, so both of these smallest elements must be the smallest element of $P_x \cap N$. Since $P = A \cup B$, this is then the smallest element of $P_x$.
    We have shown that every $P_x$ has a smallest element.
\end{proof}

\begin{corollary}[attachment]
\label{cor:gluing}
    Let $P, A, B, N$ be as above.
    Let $Y^A, Z^B$ be poset complexes. Suppose $Y^N, Z^N$ are isomorphic as $N$-complexes, with an isomorphism $f : Y^N \to Z^N$ specified. Then the quotient set
    \[
    X^P
    = Y^A \cup_{f} Z^B
    = \big( Y^A \sqcup Z^B \big)
    \big/ \big( y \sim f(y) \mid y \in Y^N \big)
    \]
    with the induced $P$-system is separated.
\end{corollary}

\begin{proof}
    This is Proposition~\ref{prop:glued} with the $P$-system defined as follows:
    \begin{alignat*}{2}
    X^a &= Y^a &\quad &\text{for $a \in A$}
    \\
    X^b &= Z^b &\quad &\text{for $b \in B$}    
    \end{alignat*}
    The definition is consistent because $Y^n = Z^n$ for $n \in N$.
\end{proof}

\begin{proposition}[topological attachment]
\label{prop:gluing-topological}
	In the situation of Corollary~\ref{cor:gluing}, suppose that $Y^A, Z^B$ are topological poset complexes, with $Y^N, Z^N$ closed subspaces, and the isomorphism $f : Y^N \to Z^N$ a homeomorphism. Assign to $X^P = Y^A \cup_f Z^B$ the quotient topology; that is, its topology as an adjunction space.
    Then each of the following adjectives applies to~$X^P$ if it applies to both $Y^A$ and $Z^B$.
    \[
    \text{
    {\rm(1)} closed
    \;
    {\rm(2)} link-cofibrant
    \;
    {\rm(3)} cellular
    \;
    {\rm(4)} compact
    \;
    {\rm(5)} Hausdorff
    }
    \]
\end{proposition}

\begin{proof}
    Since $Y^N \to Y^A$ and $Z^N \to Z^B$ are closed subspace inclusions, it follows that the canonical maps $Y^A \to X^P$ and $Z^B \to X^P$ are closed subspace inclusions.
    Any topological property in $X^P$ that is localized to $X^A$ or $X^B$ can equivalently be evaluated in $Y^A$ or $Z^B$.

    (1) For $C \subseteq P$ we have $X^C = X^{C \cap A} \cup X^{C \cap B} = Y^{C \cap A} \cup Z^{C \cap B}$, a union of two closed sets.  
    
    (2), (3) Immediate, since each $X^p$ is contained in at least one of $Y^A$ or $Z^B$.

    (4), (5) These are standard inherited properties of attachment spaces of this type.
\end{proof}

\section{The Carrier Lemma}
\label{sec:carrier-lemma}

We now have the ingredients for the generalized Carrier Lemma. We begin, in Section~\ref{subsec:poset-homotopy}, with a simple observation about homotopies of maps between spaces with poset systems. The new Carrier Lemma is presented in Section~\ref{subsec:carrier-lemma}. The preceding observation about homotopies is used to organize the proof of the second statement in the lemma. This asserts that any two continuous maps provided by the first statement in the lemma, in a given situation, must be homotopic to each other. In Section~\ref{subsec:contractibility}, we prove the stronger result that the space of all such maps is contractible (under favourable topological conditions).

\subsection{Poset-system homotopy}
\label{subsec:poset-homotopy}

Let $(X^p), (Y^q)$ be poset systems in topological spaces $X,Y$ and let $\phi : P \to Q$ be a poset map.

A $\phi$-homotopy between continuous $\phi$-maps $f,g : X \to Y$ is a homotopy for which each intermediate map is a $\phi$-map. This is the same thing as requiring the homotopy
\[
H : X \times [0,1] \longrightarrow Y
\]
to be a $\phi$-map with respect to the $P$-system
$(X^p \times [0,1])$
on the domain.
%

Alternatively, we can identify $[0,1]$ with the standard geometric $1$-simplex $|\Delta_1|$ regarded as a poset complex over the abstract 1-simplex $\Delta_1 = \{\{0\}, \{1\}, \{0,1\}\}$.
Defining
\[
\Phi
=
\phi \circ \proj_P
:
P \times \Delta_1
\longrightarrow
Q
\]
it follows that a $\phi$-homotopy is then the same thing as a $\Phi$-map
\[
H :
X \times |\Delta_1|
\longrightarrow Y
\]
with respect to the product poset system $(X^p \times |\Delta_1|^\sigma)$ in the domain.
The advantage of this second formulation is that the product structure provides additional scaffolding to build such maps, as we shall see.

\subsection{The Carrier Lemma}
\label{subsec:carrier-lemma}

The original Carrier Lemma, due to Walker~\cite{Walker_1981} building on Lakser~\cite{Lakser_1971}, provides a homotopy analogue to Eilenberg and Mac Lane's method of acyclic models in homological algebra~\cite{Eilenberg_MacLane_1953}. It gives a protocol for constructing a continuous map from a simplicial complex to a space covered by a collection of contractible subsets. Here we replace simplicial complexes by link-cofibrant poset complexes, and formulate the result in terms of extending a partially defined map (which turns out to be immediately useful, including in the proof itself). Our limitation compared to~\cite{Walker_1981} is the restriction to finite posets. 

\begin{theorem}[Carrier Lemma]
\label{thm:carrier}
Assume the following hypotheses:

    \begin{itemize}
        \item
        Let $\phi : P \to Q$ be a poset map, with $P$ finite.

        \medskip
        \item
        Let $X^P$ be a link-cofibrant $P$-complex on a topological space $X$.

        \medskip
        \item
        Let $(Y^q)$ be a cellular $Q$-system in a topological space~$Y$.

    \end{itemize}
   
    Given a closed subset $A \leq P$ (possibly empty) and a continuous $\phi|_A$-map $f_A : X^A \to Y$:

    \begin{enumerate}
        \item
        There exists a continuous $\phi$-map $f : X \to Y$ extending $f_A$.

        \medskip
        \item 
        Any two such extensions are $\phi$-homotopic relative to $X^A$.

    \end{enumerate}
\end{theorem}

\begin{remark}
    The Carrier Lemma of Walker~\cite{Walker_1981} (as stated in Section~\ref{subsec:motivation}) is recovered by setting $P = Q = K$ (the simplicial complex~$K$ viewed as a poset), $\phi = \id_K$, $X^P = |K|^K$ (the face complex) and $A = \emptyset$.
\end{remark}

Here is the inductive step in the proof of Theorem~\ref{thm:carrier}.

\begin{lemma}[Extension Lemma]
\label{lem:extension}
Let $U$ be a subspace of a topological space~$V$ such that the inclusion map $U \to V$ is a cofibration. Let $W$ be a contractible topological space.
Then any continuous map $f : U \to W$ can be extended to a continuous map $g : V \to W$.
\[
    \xymatrix{
    V
	   \ar@{.>}[r]^-{g}
    &
    W
    \\
    U
    	\ar[ur]_f
    	\ar[u]
    }
\]
\end{lemma}

\begin{proof}
Since $W$ is contractible, the map $f$ is homotopic to a map $c_U : U \to W$ that takes a constant value~$c$. Extend this trivially to the constant map $c_V: V \to Y$. Use the homotopy extension property to extend the homotopy $c_U \simeq f$ on~$U$ to a homotopy $c_V \simeq g$ on~$V$. The map $g$ is the required extension.
\end{proof}

\begin{proof}[Proof of Theorem~\ref{thm:carrier} (Carrier Lemma)]
    (1) We construct an extension cell by cell. We first extend from $A$ to $A \cup \{p\}$ where $p$ is a minimal element of $P - A$.
    The space $X^{A \cup \{p\}}$ is the union $X^A \cup X^p$ of two closed subspaces,
    and $f$ is already defined on $X^A$, so it remains to define $f$ on $X^p$ so that it agrees on the intersection, which is
    %
    \[
    X^A \cap X^p
    = X^A \cap X^{\cone(p)}
    = X^{A \cap \cone(p)}
    = X^{\link(p)}
    \]
    by separation and the minimality of~$p$.
    The task is reduced to following extension problem:
    \[
    \xymatrix{
    X^p
	   \ar@{.>}[r]
    &
    Y^{\phi(p)}
    \\
    X^{\link(p)}
    	\ar[ur]
    	\ar[u]
    }
    \]
    Since $X^{\link(p)} \to X^p$ is a cofibration and $Y^{\phi(p)}$ is contractible, a continuous extension is provided by Extension Lemma~\ref{lem:extension}. 
    We extend to all of $X^P$ by repeating this process finitely many times.
    
    (2) This is a special case of (1) in the following way. Let $f,g$ be the two extensions. Write
    \begin{alignat*}{2}
        &R &&= P \times [0,1]
        \\
        &Z^R &&= X^P \times [0,1]^{\Delta_1}
        \\
        &B
        &&=
        (P \times \partial \Delta_1)
        \;\cup\;
        (A \times \Delta_1)
    \end{alignat*}
    and define $\Phi : R \to Q$ by
    $
    \Phi(p,\sigma) = \phi(p)
    $.
    Here $Z^R$ is link-cofibrant, being a product of link-cofibrant poset complexes.
    Define $f_{B}: Z^B \to Y$ by the following piecewise formula on two overlapping open subsets:
    \[
    f_{B}(x,t)
    =
    \begin{cases}
        f(x) & \text{for $t \in [0,1)$}
        \\
        g(x) & \text{for $t \in (0,1]$}
    \end{cases}
    \]
    This is consistent on the intersection $X^A \times (0,1)$ because $f(x) = g(x) = f_A(x)$ on $X^A$, so we get a continuous map $f_B : Z^B \to Y$ which is a $\Phi|_{B}$-map because $f, g$ are $\phi$-maps.
    The extension to a ${\Phi}$-map $Z^R \to Y$ provided by (1) is the homotopy we seek.    
\end{proof}

\begin{remark}
    In the original Carrier Lemma~\cite{Walker_1981} where the domain is a simplicial complex, the proof of (2) is very similar to the proof of (1) but doesn't quite reduce to it, so the steps have to be written out again. Our framework has products, and the lemma is formulated in terms of extending a partially defined map, so we get (2) directly from (1). 
\end{remark}

\subsection{Contractibility of Mapping Spaces}
\label{subsec:contractibility}

Theorem~\ref{thm:carrier} can be interpreted heuristically as the claim that the space of continuous $\phi$-maps extending $f_A$ is (1) non-empty, and (2) path-connected. The next theorem makes this precise and strengthens the claim.

\begin{theorem}[mapping-space contractibility]
\label{thm:contractibility}
    Under the hypotheses of Theorem~\ref{thm:carrier}, suppose further that $X$ is locally compact.\footnote{%
        We take `locally compact' to mean that every point has a basis of compact neighbourhoods.
    }
    Endow the mapping space
    \[
    \map(X,Y)
    =
    \{ \text{continuous maps $X \to Y$} \}
    \]
    with the compact-open topology.
    Then the subspaces\footnote{%
    This notation suppresses the poset systems $(X^p), (Y^q)$ to reduce clutter.
    }
    \begin{align*}
        \map_{\phi}(X,Y)
        &=
        \{ \text{continuous $\phi$-maps $X \to Y$} \}
        \\
        \map_\phi(X,Y; f_A)
        &=
        \{ \text{continuous $\phi$-maps $X \to Y$ extending $f_A$} \}
    \end{align*}
    are contractible. 
\end{theorem}

We prove the second of these assertions, the first being the special case $A = \emptyset$.

\begin{lemma}
    \label{lem:contractibility}
    Let $M$ be a topological space and $L \subseteq M$ be a closed subspace whose inclusion map is a cofibration. In the situation of Theorem~\ref{thm:contractibility}, any continuous map $g : L \to \map_\phi(X,Y; f_A)$ can be extended to a continuous map $M \to \map_\phi(X,Y; f_A)$:
    \[
    \xymatrix{
    M
        \ar@{.>}[dr]^{\exists}
    \\
    L
        \ar[u]
        \ar[r]_-g
    &
    \map_\phi(X,Y; f_A)
    }
    \]
\end{lemma}

\begin{proof}
Recall that the exponential adjunction map for topological spaces
\[
\xymatrix{
\map(W \times X, Y)
    \ar[r]
&
\map(W, \map(X,Y))
}
\]
is a bijection
when $X$ is locally compact (\cite[Chapter 7, Corollary 5.4]{Munkres_1984}; stated for locally compact Hausdorff but the proof does not use Hausdorff).
This allows us to recast the extension problem in the following form
\[
    \xymatrix{
    M \times X
        \ar@{.>}[dr]^{\exists}
    \\
    L \times X
        \ar[u]
        \ar[r]_-{G}
    &
    Y
    }
\]
with appropriate restrictions on the maps to~$Y$. Here $G(\ell,x) = [ g(\ell) ](x)$.

Let $I = (0 < 1)$ be the two-element ordered set. Then $M^0 = L$, $M^1 = M$ defines $M^I$ as a link-cofibrant $I$-complex and the product $M^I \times X^P$ as a link-cofibrant $(I \times P)$-complex.

To construct the desired extension, we apply the Carrier Lemma with the following set-up.
Let $B = \{0\} \times P \;\cup\; I \times A  \leq I \times A$.
Let $\Phi = \phi \circ \proj_P : I \times P \to Q$.
Let $F_B : (M \times X)^B \to Y^Q$ be the continuous $\Phi|_B$-map defined piecewise on two closed subspaces as follows:
\[
F(m,x)
=
\begin{cases}
    G(m,x) &\text{if $m \in L$}
    \\
    f_A(x) &\text{if $x \in X^A$}
\end{cases}
\]
Any extension to a continuous $\Phi$-map $F : M^I \times X^P \to Y^Q$ has the properties we seek.
\end{proof}

\begin{proof}[Proof of Theorem~\ref{thm:contractibility}]
    Apply Lemma~\ref{lem:contractibility} with $L = \map_\phi(X,Y; f_A)$, with $g$ the identity map, and with $M$ the cone on $L$ (compare Proposition~\ref{prop:ccce}).
\end{proof}

This short proof works because we have imposed no restrictions on the underlying space of a link-cofibrant poset complex. In particular, it can be a rather complicated mapping space.

\begin{remark}
\label{rem:k-spaces}
    Local compactness in Theorem~\ref{thm:contractibility} is used to ensure that the exponential adjunction map is a bijection. A natural alternative is to work in the category of k-spaces (compactly generated spaces). One complication is that subspaces of k-spaces are not necessarily k-spaces. I leave it to the reader to investigate this approach.
\end{remark}

\section{CL-complexes}
\label{sec:CL-complexes}

We have reached the main objects of study.
Section~\ref{subsec:CL} defines CL-complexes and records their basic formal properties. Illustrative examples of brisk CL-complex arguments are given in Section~\ref{subsec:short-proofs}.
The existence of CL-complexes is held back to Section~\ref{subsec:CL-constructions}, which offers constructions both specific and general, beginning with the CL-complex structure on the order complex of a finite poset. Section~\ref{subsec:CL-further} gives examples of `motivated construction': what is the `simplest' way to realize a particular kind of poset as a CL-complex? Finally, the overall situation is laid out in categorical terms in Section~\ref{subsec:equivalence}.

\subsection{CL-complexes}
\label{subsec:CL}

A \emph{CL-complex}\footnote{%
Here C = Cellular and L = Link-cofibrant, but also CL = Carrier Lemma.
} is a cellular link-cofibrant complex over a finite poset.

\begin{proposition}[existence]
    \label{prop:existence}
    There exist CL-complexes over every finite poset $P$.
\end{proposition}

We postpone the proof to Section~\ref{subsec:CL-constructions}.
Meanwhile, combining Examples \ref{ex:closed}, \ref{ex:LC}, \ref{ex:cellular}:

\begin{example}
\label{ex:CL}
The following are CL-complexes:

(1) The face complex of a finite simplicial complex or finite regular CW-complex.

(2) Any subcomplex of a CL-complex.

(3) The product of two CL-complexes.
\qed
\end{example}

Here are the main consequences of the Carrier Lemma for CL-complexes.

\begin{theorem}
\label{thm:CL-map}
    Let $X^P, Y^Q$ be CL-complexes over finite posets $P,Q$. Let $\phi : P \to Q$ be a poset map. Then there is a unique $\phi$-homotopy class of continuous $\phi$-maps $f : X^P \to Y^Q$.
\end{theorem}

\begin{proof}
    Immediate from Theorem~\ref{thm:carrier} (Carrier Lemma).
\end{proof}

\begin{proposition}
\label{prop:unique-H-type}
    Let $X^P, Y^P$ be CL-complexes over a finite poset~$P$. Then the spaces $X, Y$ are canonically homotopy equivalent. Indeed, any continuous $P$-map $f : X^P \to Y^P$ is a homotopy equivalence, and any continuous $P$-map $g : Y^P \to X^P$ is its homotopy inverse.
\end{proposition}

\begin{proof}
    Theorem~\ref{thm:CL-map} provides continuous $P$-maps
    \[
    f : X^P \to Y^P, \quad g : Y^P \to X^P
    \]
    that are canonically defined up to homotopy. Then $gf, \id_X : X \to X$ are continuous $P$-maps and hence homotopic to each other. Likewise $fg, \id_Y : Y \to Y$ are continuous $P$-maps and hence homotopic to each other. Thus $f,g$ are homotopy inverse to each other.
\end{proof}

\begin{proposition}
\label{prop:unique-H-class}
    Let $\phi: P \to Q$ be a map of finite posets. Suppose
    \[
    \xymatrix{
    X_1^P
        \ar[r]^{f_1}
    &
    Y_1^Q
    &
    X_2^P
        \ar[r]^{f_2}
    &
    Y_2^Q
    }
    \]
    are continuous $\phi$-maps of CL-complexes. Then $f_1, f_2$ are canonically homotopy equivalent in the following sense: given any continuous $P$-map $\xi : X_1^P \to X_2^P$ and any continuous $Q$-map $\eta : Y_1^Q \to Y_2^Q$, the diagram
    \[
    \xymatrix{
    X_1^P
        \ar[r]^{f_1}
        \ar[d]_{\xi}
    &
    Y_1^Q
        \ar[d]^{\eta}
    \\
    X_2^P
        \ar[r]_{f_2}
    &
    Y_2^Q
    }
    \]
    commutes up to homotopy. If one of $f_1, f_2$ is a homotopy equivalence, then so is the other.
\end{proposition}

Note that $\eta, \xi$ are homotopy equivalences by Proposition~\ref{prop:unique-H-type}.

\begin{proof}
For the first statement, $\eta f_1$ and $f_2 \xi$ are continuous $\phi$-maps and therefore homotopic to each other. For the second statement, if three sides of a homotopy-commutative square are homotopy equivalences then so is the fourth.
\end{proof}

Propositions \ref{prop:unique-H-type} and~\ref{prop:unique-H-class} ensure that the following convention is meaningful.

\begin{convention}[poset homotopy]
\label{conv:poset-homotopy}
    We adopt the following language.
%
    The `homotopy type' of a finite poset $P$ is understood to mean the homotopy type of any CL-complex $X^P$.
%
    The `homotopy class' of a map $\phi : P \to Q$ between finite posets is understood to mean the homotopy class of any continuous $\phi$-map $f : X^P \to Y^Q$ between CL-complexes.
\end{convention}

We shall see that this agrees with the standard convention: this assigns to a poset the homotopy type of its order complex, and to a poset map the homotopy class of the induced simplicial map between order complexes. See Corollary~\ref{cor:conventions-agree}.

\begin{remark}
\label{rem:proof-strategies}
    The proof strategies (S1)--(S5) in Section~\ref{subsec:motivation} are to be understood in terms of Convention~\ref{conv:poset-homotopy}. The reader may confirm their validity using Propositions \ref{prop:unique-H-type} and~\ref{prop:unique-H-class}. 
\end{remark}

\subsection{Pithy proofs}
\label{subsec:short-proofs}

Here are some short CL-complex proofs, assuming existence.

\begin{example}[up-cone]
\label{ex:concise-up-cone}
    Let $P$ be a finite poset with a smallest element $0$. Then $P$ is contractible. Indeed, let $\{*\}$ be a singleton space and let $\{*\}^p = \{*\}$ for all~$p$. Then $\{*\}^P$ is a singleton CL-complex for~$P$. For separation, note that $P_* = P$ has smallest element~$0$.
    \end{example}

\begin{example}[down-cone]
\label{ex:concise-down-cone}
    Let $P$ be a finite poset with a largest element $\infty$. Then $P$ is contractible. Indeed, let $X^P$ be any CL-complex for~$P$. Then $X = X^P = X^{\cone(\infty)} = X^\infty$ is contractible, by cellularity.
\end{example}

\begin{example}[comparable maps are homotopic]
\label{ex:concise-comparable}
    Let $\phi, \psi : P \to Q$ be maps of finite posets. Suppose $\phi \leq \psi$. Then $\phi, \psi$ are homotopic. Indeed, let $f : X^P \to Y^Q$ be a continuous $\phi$-map of CL-complexes. Then $f$ is a $\psi$-map of CL-complexes.
\end{example}

\begin{example}[product of two posets]
\label{ex:concise-product}
    Let $P, Q$ be finite posets. Then the homotopy type of $P \times Q$ is the product of the homotopy types of $P$ and $Q$. Indeed, if $X^P, Y^Q$ are CL-complexes for $P,Q$ then $X^P \times Y^Q$ is a CL-complex for $P \times Q$.
\end{example}

\begin{example}[regular CW-complexes]
\label{ex:concise-regular-CW}
    A finite regular CW-complex~$X$ is homotopy equivalent to its face poset~$A$. Indeed, the face complex~$X^A$ is a CL-complex for~$A$.
\end{example}

The conciseness is the point. Other simple proofs are readily available. For example, it is not difficult to show directly that a regular CW-complex is homeomorphic (more than just homotopy equivalent) to the order complex on its face poset.

\subsection{Existence}
\label{subsec:CL-constructions}

It is time to build CL-complexes.
Here is the quickest construction.

\begin{construction}[order complexes as CL-complexes]
\label{construction:order}
    Let
    \[
        X = |\order(P)|
    \]
    be the geometric realization of the order complex of~$P$. Define a $P$-system by
    \[
        X^p = |\order(\cone(p))|.
    \]
%
    To show that $(X^p)$ is separated, let $|\sigma|$ be the unique geometric simplex that contains a given point $x$ in its interior. We have $x \in X^p$ iff $\sigma \in \order(\cone(p))$ iff $\max(\sigma) \leq p$, so $P_x$ has smallest element $\max(\sigma)$ as required. Each $X^p$ is a geometric cone and a simplicial subcomplex, so $X^P$ is cellular and link-cofibrant.
Thus $X^P = |\order(P)|^P$ is a CL-complex.
\end{construction}

\begin{proposition}[maps of order complexes]
\label{prop:order-map}
    Let $\phi: P \to Q$ be a map of finite posets. Let $|\phi| : |\order(P)| \to |\order(Q)|$ denote the geometric realization of the induced map of order complexes. Then $|\phi|$ is a $\phi$-map with respect to the `cone' CL-complex structures on $|\order(P)|$, $|\order(Q)|$ defined in Construction~\ref{construction:order}.
\end{proposition}

\begin{proof}
    For $p \in P$, we have $\phi(\cone(p)) \subseteq \cone(\phi(p))$ by monotonicity of~$\phi$. It follows that
    $
    |\phi|(|\cone(p)|)
    \subseteq
    |\cone(\phi(p))|
    $ by the functoriality of the order complex.
\end{proof}

Construction~\ref{construction:order} and Proposition~\ref{prop:order-map} immediately imply:

\begin{corollary}
\label{cor:conventions-agree}
    The CL-complex homotopy type of a finite poset is the same as its order-complex homotopy type. The CL-complex homotopy class of a map of finite posets is the same as its order-complex homotopy class.
    \qed
\end{corollary}

More generally, we can build CL-complexes using topological attachment.

\begin{theorem}[existence]
\label{thm:existence}
    Let $P$ be a finite poset, let $A \leq P$, and let $Y^A$ be a CL-complex. Then there exists a CL-complex $X^P$ that extends $Y^A$ in the sense that the construction comes with an isomorphism $Y^A \to X^A$ of topological $A$-complexes. If $Y^A$ is compact or Hausdorff or both, we can construct $X^P$ to be compact or Hausdorff or both.
\end{theorem}

\begin{proof}
    We may assume $A \ne P$. Let $p$ be a minimal element of $P-A$, and write $B = \cone(p)$ and $N = A \cap B = \link(p)$.
    By Proposition~\ref{prop:ccce}, there exists a closed cofibrant inclusion $f : Y^{N} \to Z$, for some contractible space~$Z$. Define $Z^b = f(Y^b)$ for $b \in N$, and $Z^p = Z$. Then $Z^B$ is a CL-complex, with $f$ restricting to an isomorphism $Y^N \to Z^N$ of topological $N$-complexes. By Proposition~\ref{prop:gluing-topological} (topological attachment) the space
    \[
    X^{A \cup B}
    =
    Y^A \cup_f Z^B
    \]
    has the structure of a CL-complex for $A \cup B = A \cup \{p\}$.
    Repeat until done.
    The `compact or Hausdorff or both' assertion follows from Propositions \ref{prop:ccce} and~\ref{prop:gluing-topological}.
\end{proof}

We extract an iterative construction (with choices) from the inductive proof above:

\begin{construction}
\label{construction:iterative}
    Let $X^A$ be a CL-complex for some $A \leq P$.
    Let $p_1, \dots, p_n$ list the elements of $P-A$ in a sequence that is non-decreasing with respect to the partial order. Let $A_k = A \cup \{p_1, \dots, p_k\}$ with $A_0 = A$. Iteratively, set
    \[
    X^{p_k} = Z_k,
    \quad
    X^{A_{k}} = X^{A_{k-1}} \cup_{f_k} X^{p_k} 
    \]
    where $f_k, Z_k$ are chosen such that
    $
    f_k : X^{\link(p_k)} \to Z_k
    $
    is a closed cofibrant inclusion into a contractible space.
    Then $X^P$ is a CL-complex.
\end{construction}

The procedure above yields (up to isomorphism) {every} possible CL-complex extending the starting complex.
Indeed, in any such extension we have the relationship $X^{A_k} = X^{A_{k-1}} \cup X^{p_k}$ where $X^{p_k}$ is a contractible space containing $X^{\link(p_k)}$ by a closed cofibrant inclusion map.
The freedom lies in the choice of such an inclusion at each step in the iteration.

Suppose we instantiate Construction~\ref{construction:iterative} in the simplest possible way, starting with $X^\emptyset = \emptyset$ and always taking the cone inclusion.
This gives the following direct construction.

\begin{construction}
\label{construction:join}
    Identify $P$ with an affinely independent subset of a real vector space. Construct a poset system $(X^p)$ by defining
    \[
    X^p
    =
    \{ p \} \star \bigcup_{q < p} X^q
    =
    \{p\} \star X^{\link(p)} 
    \]
    recursively with respect to the partial order,\footnote{%
    Here $\star$ is the geometric join, with $A \star B$ being the union of line segments from $A$ to $B$.
    }
    and let $X = X^P$.
\end{construction}

In fact, Constructions \ref{construction:order} and~\ref{construction:join} are isomorphic, as the reader may readily verify by induction (or otherwise).
The first construction supposes prior knowledge of order complexes.
The second construction is the inevitable consequence of building CL-complexes iteratively with minimal effort. From this point of view, order complexes emerge automatically as the simplest possible CL-complexes: they invent themselves.

\subsection{Further constructions}
\label{subsec:CL-further}

If a poset has a particular structure, we can sometimes use that structure to guide the construction of a CL-complex, differently from constructing the order complex.
In the following example $\Delta(I)$ denotes the complete simplicial complex on a set~$I$. This is the set of nonempty finite subsets of~$I$.

\begin{example}[simplicial complex]
\label{ex:lazy-sc}
    Let $K$ be a simplicial complex on a finite set~$I$.
    Thinking of it as a poset $K \leq \Delta(I)$,
    we construct a CL-complex $X^K$ as lazily as possible. The minimal elements are the singleton sets\footnote{%
    For this example we use square brackets to delimit the list of distinct elements of a subset.}
    $[a]$, so the $X^{[a]}$ must be disjoint contractible sets. The lazy choice is to make them singletons. For $\sigma \in K$, we require $X^\sigma$ to be contractible and contain $X^{[a]}$ for all $a \in \sigma$. The lazy choice is to take the convex hull. This leads to the following construction. Identify $I$ with an affinely independent subset of a real vector space. For each $\sigma = [a_0, a_1, \dots, a_k] \in K$, define:
    \[
    X^\sigma
    = \conv(a_0 , a_1, \dots, a_k)
    = a_0 \star a_1 \star \dots \star a_k
    \]
    Through pure indolence we have recovered the face complex $|K|^K$ of the simplicial complex.
\end{example}

\begin{remark}
    This, {too}, is isomorphic to the cone CL-structure on the order complex. The barycentric subdivision homeomorphism $|\order(K)| \to |K|$ carries $|\order(\cone(\sigma))|$ to the geometric simplex~$|\sigma|$, so that $\ocl(K) = |\order(K)|^K$ and $|K|^K$ are topologically isomorphic CL-complexes.
    The difference is motivational: in $|K|^K$ we directly recognize $|K|$ as a topological realization of $K$ without invoking barycentric subdivision.
\end{remark}

\begin{remark}
\label{rem:sc-homotopy-type}
    The `homotopy type' of a simplicial complex~$K$ is unambiguously defined, independently of whether we regard it as a simplicial complex or as a poset. We are free to understand it to mean the homotopy type of the geometric realization~$|K|$ of the simplicial complex~$K$, the homotopy type of the geometric realization of the order complex of the poset~$K$, or, in the finite case, the homotopy type of any CL-complex for the poset~$K$. All three interpretations are canonically homotopy equivalent as indicated above.
\end{remark}

\begin{example}[prod-complexes \cite{deSilva_GIRSSVY_2026}]
\label{ex:prod-complexes}
    Let $I, J$ be finite sets and let $A \leq \Delta(I) \times \Delta(J)$. We can realize $A$ as a CL-subcomplex of a product of complete simplicial complexes
    \[
    |\prodc(A)|^A
    =
    (|\Delta(I)| \times |\Delta(J)|)^A
    \leq
    |\Delta(I)|^{\Delta(I)} \times |\Delta(J)|^{\Delta(J)}
    \]
    with cells $|\prodc(A)|^{(\sigma,\tau)} = |\sigma| \times |\tau|$.
    This object with its cell structure is called a `{prod-complex}' in~\cite{deSilva_GIRSSVY_2026}.
    An alternative construction produces a simplicial complex: let $\mu : \Delta(I) \times \Delta(J) \to \Delta(I \times J)$ denote the poset map $(\sigma,\tau) \mapsto \sigma \times \tau$, then
    \[
    \simp(A) = \down{\,\mu(A)} \leq \Delta(I \times J)
    \]
    is a simplicial complex on $I \times J$, called the `{simplexification}' of~$A$. Then $|\simp(A)|$ admits two natural CL-complex structures: its face complex structure $|\simp(A)|^{\simp(A)}$, and the structure $|\simp(A)|^A$ defined\footnote{%
    Verifying separation takes a little thought: for $x$ in the relative interior of $|\rho|$, one shows that $x \in |\sigma \times \tau|$ iff $(\sigma, \tau) \geq (\rho_I, \rho_J)$ where $\rho_I, \rho_J$ are the projections of $\rho$ onto $\Delta(I), \Delta(J)$.
    } by the formula
    \[
    |\simp(A)|^{(\sigma,\tau)} = |\sigma \times \tau|
    \]
    for $(\sigma,\tau) \in A$.
    To relate the two geometric realizations, a continuous `squashing' map
    \[
    \xi : |\simp(A)| \to |\prodc(A)|
    \]
    is defined in~\cite{deSilva_GIRSSVY_2026} and shown using Smale's mapping theorem to be a homotopy equivalence.
    For a shorter proof (in the finite case) note that $\xi : |\simp(A)|^A \to |\prodc(A)|^A$ is an $A$-map of CL-complexes and hence a homotopy equivalence.
\end{example}

\subsection{Equivalence of categories}
\label{subsec:equivalence}

We consider the situation in terms of categories.

A \emph{CL-map} between CL-complexes $X^P, Y^Q$ over finite posets $P,Q$ is a pair $(f,\phi)$ where $\phi : P \to Q$ is a poset map and $f : X \to Y$ is a continuous $\phi$-map.
Two CL-maps $(f,\phi), (g,\psi) : X^P \to Y^Q$ are \emph{CL-homotopic} if $\phi = \psi$ and $f,g$ are $\phi$-homotopic.

Let $\cl$ denote the category of CL-complexes and CL-maps. Let $\hcl$ denote the category of CL-complexes and CL-homotopy classes of CL-maps. Let $\topcat$ denote the category of topological spaces and continuous maps. Let $\htop$ denote the category of topological spaces and homotopy classes of maps. Let $\fposet$ denote the category of finite posets and poset maps.
These categories fit into the following diagram, which commutes exactly:
\[
\xymatrix{
\topcat
    \ar[d]
&
\cl
    \ar[l]
    \ar[r]
    \ar[d]
&
\fposet
    \ar@{=}[d]
\\
\htop
&
\hcl
    \ar[l]
    \ar[r]
&
\fposet
}
\]
The horizontal arrows are the forgetful functors: from a CL-complex $X^P$ we extract the topological space~$X$ and the finite poset~$P$, and from a CL-map $(f,\phi)$ we extract the continuous map~$f$ and the poset map $\phi$. The vertical arrows fix all objects and take each morphism to its homotopy class.

\begin{proposition}
\label{prop:hCL=fPoset}
    The forgetful functor $\hcl \longrightarrow \fposet$ is an equivalence of categories.
\end{proposition}

\begin{proof}
    It is surjective on objects by Proposition~\ref{prop:existence}. It is full and faithful on morphism sets by Theorem~\ref{thm:CL-map}. 
\end{proof}

Order complexes fit into the picture in the following way.

\begin{remark}[inverse equivalence]
\label{rem:inverse-functor}
Let $\ocl(P) = |\order(P)|^P$ denote the geometrically realized order complex of a finite poset~$P$ with the `cones' CL-complex structure described above.
Given a poset map $\phi : P \to Q$, let $|\phi| : |\order(P)| \to |\order(Q)|$ be the induced map. Then $\ocl(\phi) = (|\phi|, \phi)$ is a CL-map by Proposition~\ref{prop:order-map}.
In this way, we get a functor $\ocl$ from finite posets to CL-complexes.
The composite
\[
\xymatrix{\fposet \ar[r]^-{\ocl} & \cl \ar[r] & \hcl}
\]
is an inverse equivalence (indeed, a strict right-inverse) to the equivalence
\[
\xymatrix{\hcl \ar[r] & \fposet}
\]
of Proposition~\ref{prop:hCL=fPoset}.
\end{remark}

\section{Fiber Theorems}
\label{sec:fiber-theorems}

Fiber theorems give locally testable conditions that a given map is a homotopy equivalence.
Here we consider three such theorems.
Section~\ref{subsec:QTA} gives a CL-complex proof of Quillen's Theorem~A in the special case of finite posets.
We encounter the fundamental asymmetry of CL-complexes with respect to poset reversal: only one of the two dual formulations has a natural proof using CL-complexes.
Quillen's theorem has a relative version which we discuss in Section~\ref{subsec:QTA-relative}, again for finite posets, using the formalism of link-cofibrant complexes.
Section~\ref{subsec:product-fiber} concerns subcomplexes of a product complex: if the projection of the subcomplex onto one of the factors has contractible fibers, then it is a homotopy equivalence. The natural CL-complex argument requires a verification that seemingly depends on stronger results that already imply the whole theorem. We present the argument anyway.

\subsection{Quillen's Theorem~A}
\label{subsec:QTA}

In his development of higher algebraic K-theory~\cite{Quillen_1973}, Quillen obtained, as a useful tool, a sufficient criterion for a functor between two categories to induce a homotopy equivalence of classifying spaces. This is the famous `Quillen's Theorem~A'.

CL-complexes enable a very short proof in the case where the two categories are finite posets.
The reader may wish to compare Walker's short proof~\cite[Theorem 2.2]{Walker_1981} for arbitrary posets, based on his Carrier Lemma.

\begin{theorem}[QTA for finite posets]
\label{thm:QTA}
Let $\phi : P \to Q$ be a map between finite posets. Suppose each Quillen fiber
$\phi/q = \phi^{-1}(\cone(q))$
is a contractible subposet of $P$. Then $\phi$ is a homotopy equivalence.
\end{theorem}

\begin{proof}
    Let $X^P$ be a CL-complex. The Quillen fiber hypothesis implies that the $\phi$-pushforward $X^Q$ is a CL-complex: for cellularity, each $X^q = X^{\phi/q}$ is a CL-complex for $\phi/q$ and therefore contractible. Then the identity homeomorphism $\id_X : X^P \to X^Q$ is a $\phi$-map of CL-complexes (Example~\ref{ex:phi-pushforward}) so $\phi$ is a homotopy equivalence of posets.
\end{proof}

Here is the dual statement. Using standard methods this would be a trivial consequence.

\begin{theorem}[QTA for finite posets, dual form]
\label{thm:QTA-dual}
Let $\phi : P \to Q$ be a map between finite posets. Suppose each Quillen up-fiber
$q \backslash \phi= \phi^{-1}(\up{\{q\}})$
is a contractible subposet of $P$. Then $\phi$ is a homotopy equivalence.
\end{theorem}

CL-complexes do not help here. Knowledge of the up-fibers doesn't give us any useful information about a generic CL-complex~$X^P$, because the subsystem $(X^a \mid a \in A)$ associated to an up-closed set $A \subseteq P$ is not a subcomplex and we have no real control over it.
Instead we formally invoke the poset-reversal symmetry of order complexes.

\begin{proof}
    The poset map $\phi : P \to Q$ is simultaneously  a poset map $\op{\phi} : \op{P} \to \op{Q}$.
    Then $\op{\phi}$ is a homotopy equivalence by Theorem~\ref{thm:QTA}, meaning that the geometric realization
    \[
    |\op{\phi}| : |\ocl(\op{P})| \longrightarrow |\ocl(\op{Q})|
    \]
    is a homotopy equivalance.
    From the commutative diagram
    \[
    \xymatrix{
    |\order(P)|
        \ar[rr]^{|\phi|}
        \ar@{=}[d]
    &&
    |\order(Q)|
        \ar@{=}[d]
    \\        
    |\order(\op{P})|
        \ar[rr]^{|\op{\phi}|}
    &&
    |\order(\op{Q})|
    }
    \]
    it follows that $|\phi|$ is a homotopy equivalence, hence $\phi$ is a homotopy equivalence of posets.
\end{proof}

\subsection{Relative QTA}
\label{subsec:QTA-relative}

Quillen's Theorem~A has a relative form whose statement is easily guessed. A high-level proof appears in Cisinski \cite[Theorem 2.1.13]{Cisinski_2004}.
Here is the statement for finite posets.

\begin{theorem}[relative QTA for finite posets]
\label{thm:QTA-rel}
    Consider a diagram of finite posets:
    \[
    \xymatrix{
    P
        \ar[r]^-{\phi}
        \ar[dr]_-{\alpha}
    &
    Q
        \ar[d]^-{\beta}
    \\
    &
    R
    }
    \]
    Suppose the induced map of Quillen fibers $\alpha/r \longrightarrow \beta/r$ is a homotopy equivalence for all~$r$.
    Then $\phi$ is a homotopy equivalence.
\end{theorem}

Theorem~\ref{thm:QTA} is the special case of Theorem~\ref{thm:QTA-rel} by setting $R = Q$ and $\beta = \id_Q$. The proof of the relative statement depends on a deeper understanding of link-cofibrant complexes. We will reduce the theorem to the following:

\begin{theorem}[Generalized Homotopy Gluing Theorem]
    \label{thm:L-complex-HE}
    Let $R$ be a finite poset, let $X^R, Y^R$ be link-cofibrant $R$-complexes, and let $f : X^R \to Y^R$ be a continuous $R$-map which restricts to a homotopy equivalence $f^r : X^r \to Y^r$ for all $r \in R$. Then $f$ is a homotopy equivalence.
\end{theorem}

This result is an inductive extension of a homotopy gluing theorem which we lift from Brown~\cite{Brown_2006} with some notation changes:

\begin{theorem}[Homotopy Gluing Theorem {\cite[7.4.3]{Brown_2006}}]
\label{thm:gluing}
Suppose $X, Y$ are topological spaces with closed subspaces $X^i, Y^i$ ($i = 0, 1, 2$) such that:
    \begin{align*}
        X &= X^1 \cup X^2
        &
        Y &= Y^1 \cup Y^2
        \\
        X^0 &= X^1 \cap X^2
        &
        Y^0 &= Y^1 \cap Y^2
    \end{align*}
    Suppose $f : X \to Y$ is a continuous map that restricts to continuous maps
    \[
    f^0 : X^0 \to Y^0,
    \quad
    f^1 : X^1 \to Y^1,
    \quad
    f^2 : X^2 \to Y^2
    \]
    and suppose that these maps are homotopy equivalences.
    Finally, suppose that the inclusion maps
    \begin{align*}
    X^0 &\longrightarrow X^1
    &
    Y^0 &\longrightarrow Y^1
    \\
    X^0 &\longrightarrow X^2
    &
    Y^0 &\longrightarrow Y^2
    \end{align*}
    are cofibrations. Then $f : X \to Y$ is a homotopy equivalence.
    \qed
\end{theorem}

The proof in~\cite{Brown_2006} is more subtle than the clean theorem statement might suggest.

\begin{remark}
\label{rem:gluing-V}
Let $V = (1 > 0 < 2)$ denote the V-shaped poset of Example~\ref{ex:V-complex}. Then the Homotopy Gluing Theorem~\ref{thm:gluing} is the special case of Theorem~\ref{thm:L-complex-HE} with $R = V$.
%
\end{remark}

\begin{proof}[Proof of Theorem~\ref{thm:L-complex-HE}]
        We proceed by induction on the cardinality of $R$. The result is trivial if $R$ is empty.
    Suppose that $R$ is not empty and that the result has been proved for all posets of smaller cardinality.
    Let $a \in R$ be a maximal element and let $V$ be the V-shaped poset. 
    We have a poset map $\psi : R \to V$ defined as follows:
    \[
    \psi(r) =
    \begin{cases}
        1 &\text{if $r = a$}
        \\
        0 &\text{if $r \in \link(a)$}
        \\
        2 &\text{if $r \in R - \cone(a)$}
    \end{cases}
    \]
    The Quillen fibers of this map are:
    \begin{align*}
        \psi/1 &= \cone(a)
        &
        \psi/0 &= \link(a)
        &
        \psi/2 &= R - \{a\}
    \end{align*}
    Let $X^V, Y^V$ be the $\psi$-pushforwards of $X^R, Y^R$.
    These are link-cofibrant by Example~\ref{ex:LC}~(4), and we have:
    \begin{align*}
        X^1 &= X^{\cone(a)} = X^a
        &
        X^0 &= X^{\link(a)}
        &
        X^2 &= X^{R - \{a\}}
        \\
        Y^1 &= Y^{\cone(a)} = Y^a
        &
        Y^0 &= Y^{\link(a)}
        &
        Y^2 &= Y^{R - \{a\}}
    \end{align*}
    The map $f^1 = f^a : X^a \to Y^a$ is a homotopy equivalence by assumption, while $f^0, f^2$ are homotopy equivalences by the inductive hypothesis.
    The Homotopy Gluing Theorem~\ref{thm:gluing} implies that $f$ is a homotopy equivalance. 
\end{proof}

Proposition~\ref{prop:push-square}, with $R = S$ and $\psi = \id_R$, provides the setting for the following proof.

\begin{proof}[Proof of Theorem~\ref{thm:QTA-rel}]
    Let $f : X^P \to Y^Q$ be a continuous $\phi$-map of CL-complexes.
    Let $X^R, Y^R$ denote the $\alpha$- and $\beta$-pushforwards of~$X^P, Y^Q$.
    Then $f : X^R \to Y^R$ is an $R$-map of link-cofibrant $R$-complexes. Moreover $X^r = X^{\alpha/r}$ and $Y^r = Y^{\beta/r}$ have the structure of CL-complexes for $\alpha/r$, $\beta/r$ and the restriction  $f^r : X^r \to Y^r$ realizes the induced map $\alpha/r \to \beta/r$. 
    By hypothesis, each $f^r$ is a homotopy equivalence, so
    $f$ is a homotopy equivalence by Theorem~\ref{thm:L-complex-HE}.
\end{proof}

A result very similar to Theorem~\ref{thm:L-complex-HE} is given by tom~Dieck \cite[Theorem~1]{tomDieck_1971} for a map of topological spaces over a base space with a numerable cover (a cover that admits a locally finite partition of unity). 
The natural way to cast our result in that form doesn't  work, because the canonical cover of a CL-complex is not, in general, numerable: closed covers need not admit partitions of unity (consider $[0,2] = [0,1] \cup [1,2]$).

\subsection{Product fiber theorem}
\label{subsec:product-fiber}

In this section, mostly for illustrative purposes, we identify a class of CL-complexes that occur naturally when considering products of posets.

Let $P, Q$ be posets and let $A \leq P \times Q$ be a closed subset of their product. Then $A$ can be expressed in the form
\[
A = \bigcup_{q \in Q} P_q \times \{q\}
\]
for a family of closed sets $(P_q)$ such that $P_{q'} \leq P_q$ whenever $q \leq q'$.
Here is a standard consequence of QTA for posets. We use the unrestricted version \cite[Theorem 2.2]{Walker_1981} rather the finite-only Theorem~\ref{thm:QTA}.

\begin{proposition}[product fiber theorem]
\label{prop:product-fiber}
    In the situation above, suppose every `point fiber' $P_q$ is contractible. Then the projection $\pi : A \to Q$ is a homotopy equivalence.
\end{proposition}

Since contractible implies non-empty, the hypothesis implies that $\pi$ is onto.

\begin{proof}
    We use the dual form of QTA for posets. For all $q \in Q$, the Quillen up-fiber
    \[
    q \backslash \pi
    =
    \bigcup_{q' \geq q} P_{q'} \times \{q'\}
    \]
    deformation retracts to the subposet $P_{q} \times \{q\}$ through the involution $(p,q') \mapsto (p,q)$. This subposet is isomorphic to $P_{q}$ which is contractible by hypothesis.
\end{proof}

What about the Quillen down-fibers (also known simply as Quillen fibers)?

\begin{proposition}
\label{prop:product-fiber-down}
    In the situation above, each Quillen fiber
    \[
    \pi / q
    =
    \bigcup_{q' \leq q} P_{q'} \times \{q'\}
    \]
    is contractible.
\end{proposition}

Here there is no obvious retraction to $P_q \times \{q\}$ because the sets $P_{q'}$ are larger than $P_q$.

\begin{proof}
    Proposition~\ref{prop:product-fiber} for $\pi/q \leq P \times \cone(q)$ gives $\pi/q \simeq \cone(q)$ which is contractible.
\end{proof}

This puts us in a curious position. Since the Quillen fibers are contractible, this would give us a second proof of Proposition~\ref{prop:product-fiber} by (non-dualized) QTA for posets, except that we used Proposition~\ref{prop:product-fiber} to prove that the Quillen fibers are contractible.
Perhaps the right perspective, avoiding circular reasoning, is to regard the contractibility of the Quillen fibers as part of the overall architecture of the homotopy equivalence of Proposition~\ref{prop:product-fiber}.

In the finite case, we can use Proposition~\ref{prop:product-fiber-down} to realize the homotopy equivalence $\pi : A \to Q$ as a homeomorphism of CL-complexes. Again this does not give an independent proof of the homotopy equivalence, only a retrospective corroboration:

\begin{example}
\label{ex:product-fiber-CL}
Let $P, Q, A, \pi$ be as in Proposition~\ref{prop:product-fiber} and suppose that $P, Q$ are finite. Let $X^P, Y^Q$ be CL-complexes, and consider the sub-CL-complex $Z^A = (X \times Y)^A$ of $X^P \times Y^Q$.
Certainly the $\pi$-pushforward $Z^Q$ is a link-cofibrant complex. 
Moreover, each $Z^q = Z^{\pi/q}$ has the structure of a CL-complex for $\pi/q$ which is contractible by Proposition~\ref{prop:product-fiber-down}. It follows that $Z^Q$ is a CL-complex, and the identity homeomorphism $\id_Z : Z^A \to Z^Q$ is a $\pi$-map (Example~\ref{ex:phi-pushforward}) of CL-complexes.
\end{example}

The cells of $Z^Q$ provide a good example of non-cell-like contractible spaces:
\[
Z^q
=
\bigcup_{q' \leq q}
X^{P_{q'}} \times Y^{q'}
\]
The `most important' part $X^{P_q} \times Y^q$ is a product of contractible spaces and therefore contractible. Other contractible parts $X^{P_{q'}} \times Y^{q'}$ are glued onto it at the boundary, like antennae extending in the $X$ direction. The whole thing is contractible by Proposition~\ref{prop:product-fiber-down} but not entirely obviously.

\begin{remark}
    Since
    \[
    Z^q
    =
    \bigcup_{q' \leq q}
    X^{P_{q'}} \times Y^{q'}
    =
    (X \times Y)^A \cap (X^P \times Y^q)
    =
    Z^A \cap \proj_Y^{-1}(Y^q)
    \]
    it follows that $Z^Q$ is not only the $\pi$-pushforward of~$Z^A$ but also the $f$-pullback of $Y^Q$, where $f : Z^A \to Y^Q$ is the restriction of the projection map $\proj_Y : X^P \times Y^Q \to Y^Q$.
    This is a rare situation where an $f$-pullback has a good structure.
\end{remark}

\begin{example}[Dowker's theorem]
\label{ex:Dowker}
    Let $R \subseteq I \times J$ be a binary relation between sets $I,J$. Define $A \leq \Delta(I) \times \Delta(J)$ by
    \[
    A = \{ (\sigma, \tau) \mid
    \sigma \in \Delta(I) ,\,
    \tau \in \Delta(J) ,\,
    \sigma \times \tau \subseteq R
    \}
    \]
    and let $D_I = D_I(R)$ and $D_J = D_J(R)$
    denote the projections of $A$ onto $\Delta(I)$ and $\Delta(J)$.
    These simplicial complexes are the Dowker complexes~\cite{Dowker_1952} of the binary relation~$R$. It is not difficult to show that both projections
    \[
    \xymatrix{
    D_I
    &
    A
        \ar[l]
        \ar[r]
    &
    D_J
    }
    \]
    have contractible point fibers and therefore are homotopy equivalences by Proposition~\ref{prop:product-fiber}.
    This is Dowker's theorem in poset language.
    When $I,J$ are finite, we can realize this diagram as a diagram of CL-complexes in various ways. The following diagram incorporates the two realizations of~$A$ from Example~\ref{ex:prod-complexes}.
    All arrows are CL-maps and homotopy equivalences.
    \[
    \xymatrix{
    &
    |\simp(A)|^A
        \ar[dl]
        \ar[d]^{\xi}
        \ar[dr]
    \\
    |D_I|^{D_I}
    &
    |\prodc(A)|^A
        \ar[l]
        \ar[r]
    &
    |D_J|^{D_I}
    }
    \]
    Note that $|\simp(A)|$ is the `{rectangle complex}' in Brun and Salbu's proof of Dowker's Theorem~\cite{Brun_Salbu_2022}, while $|\prodc(A)|$ is the `{relational product}' in one of the proofs by Yoon~\cite{Yoon_2026}.
    A third option is to use pushforwards:
    \[
    \xymatrix{
    |\prodc(A)|^{D_I}
    &
    |\prodc(A)|^{A}
        \ar[l]
        \ar[r]
    &
    |\prodc(A)|^{D_J}
    }
    \]
    By Example~\ref{ex:product-fiber-CL}, it follows that this is a diagram of CL-complexes. From this point of view, Dowker's theorem holds because the posets $A, D_I, D_J$ are realizable by the same space.
\end{example}

\section{Nerve Theorems}
\label{sec:nerve}

The Nerve Theorem refers to a family of results \cite{Leray_1945,Borsuk_1948,Weil_1952,McCord_1967} that identify the homotopy type (or weak homotopy type, or homology) of a space covered by subsets whose nonempty finite intersections are contractible (or weakly contractible, or homologically trivial) with the homotopy type (or weak homotopy type, or homology) of a simplicial complex, introduced by Alexandrov~\cite{Alexandroff_1928}, called the \emph{nerve} of the covering.
For a comprehensive overview of the many different versions of the Nerve Theorem, I recommend Bauer et~al.~\cite{Bauer_K_R_R_2023}.

In this section, I give a short CL-complex proof of a suitably tailored version of the Nerve Theorem. This adapts immediately to a proof of Sheehy's Multicover Nerve Theorem~\cite{Sheehy_2012}.
%

\subsection{Tame classes}
\label{subsec:tame}

Define a \emph{tame class} to be a lattice\footnote{%
A family of sets that is closed under pairwise intersections and unions.%
} $\TT$ of closed sets in a topological space such that each inclusion map between elements of $\TT$ is a cofibration.

\begin{example}
The following lattices of closed sets are tame classes:

(1)
Subcomplexes of a simplicial complex or CW-complex.

(2)
Compact semi-algebraic subsets of~$\Rr^n$.

(3)
Compact definable subsets of~$\Rr^n$ in an o-minimal theory~\cite{vdDries_1998}.

(4) Finite unions of compact convex sets in~$\Rr^n$.

\end{example}

\begin{proof}
(1) is covered by Example~\ref{ex:LC}(1) and Proposition~\ref{prop:cf-subcomplexes}.

(2), (3) reduce to (1) by the triangulation theorem \cite[Chapter~8, (2.9)]{vdDries_1998}.

(4) follows from the theory of ANR spaces: a finite union of compact convex sets is an ANR space and a closed inclusion of ANR spaces is a cofibration \cite[A.6.3, A.6.6, A.6.7]{Fritsch_Piccinini_1990}.
\end{proof}

\subsection{Covers}
\label{subsec:covers}

Let $\UU = (U_i \mid i \in I)$ be a family of sets indexed by a set~$I$. 
Writing
\[
\uall = \bigcup_{i \in I} U_i
\]
we think of $\UU$ as a cover of $\uall$.
We define some related subsets:
\begin{alignat*}{2}
    U^\sigma &= \bigcap_{i \in \sigma} U_i
    &\qquad&
    \text{for $\sigma \in \Delta(I)$ \;
    ($\Leftrightarrow$
    $\sigma$ a nonempty finite subset of~$I$)}
    \\
    U^A &= \bigcup_{\sigma \in A} U^\sigma
    &&
    \text{for $A \subseteq \Delta(I)$}    
\end{alignat*}
The `Venn diagram' of intersections of the $U_i$ is encoded by the \emph{nerve} of~$\UU$~\cite{Alexandroff_1928}:
\begin{equation}
\label{eq:nerve}
N
=
N_{\UU}
=
\nerve(\UU)
=
\{
\sigma \in \Delta(I)
\,\mid\,
U^{\sigma} \ne \emptyset
\}    
\end{equation}
This is a simplicial complex because $\sigma \leq \tau$ implies $U_\tau \subseteq U_\sigma$. We also view it as a poset. We adopt the convention that $\nerve(\UU)$ indicates the simplicial complex while $N$ or $N_{\UU}$ indicates the poset.

Now suppose that $\UU$ is a family of subsets of a topological space. 

\begin{definition}
\label{def:good-cover}
     We call $\UU$ a \emph{good cover} (of~$\uall$) if every $U^\sigma$ is contractible or empty.
\end{definition}

\begin{definition}
\label{def:tame-cover}
    We call $\UU$ a \emph{tame cover} if the elements of~$\UU$ belong to some tame class~$\TT$. 
\end{definition}

This last condition is equivalent to a set of requirements intrinsic to the cover itself, without reference to any particular tame class. In practice, it is probably easier to invoke a known tame class rather than test the intrinsic requirements on an \emph{ad hoc} basis.

\begin{proposition}
\label{prop:tame}
    A family of subsets $\UU = (U_i \mid i \in I)$ of a topological space  is a tame cover iff every $U_i$ is closed and $U^A \to U^B$ is a cofibration for all $A \subseteq B \subseteq \Delta(I)$ with $A,B$ finite.
\end{proposition}

\begin{proof}
    Indeed $\UU$ is a tame cover iff the smallest lattice containing the elements of~$\UU$ is a tame class. We can identify this smallest lattice as $\{U^A \mid \emptyset \ne A \subseteq \Delta(I), \, \text{$A$ finite} \}$. When is it a tame class? The elements are closed sets iff every $U_i$ closed. The inclusions are cofibrations iff $U^A \to U^B$ is a cofibration whenever $U^A \subseteq U^B$. Ostensibly this is a broader quantification than `for all $A \subseteq B$', but in fact it is not more general, because when $U^A \subseteq U^B$ we have $U^B = U^B \cup U^A = U^{B \cup A}$ so the inclusion $U^A \to U^B$ can be written as $U^A \to U^{B \cup A}$.
\end{proof}

\subsection{The nerve theorem}
\label{subsec:nerve}

Here is the CL-complex version of the Nerve Theorem.

\begin{theorem}[nerve theorem]
\label{thm:nerve}
    Let $\UU = (U_i \mid i \in I)$ be a good tame cover in a topological space, with $I$ finite.
    Then $\uall$ is the underlying space of a CL-complex for $\op{N}$, and therefore homotopy equivalent to $|\nerve(\UU)|$.
\end{theorem}

\begin{proof}
    The $\op{N}$-system $(U^\sigma \mid \sigma \in N$) defines a CL-complex structure on~$\uall$: it is separated because $\op{N}_x = \{ \sigma \in N \mid x \in U^\sigma \}$ has the $\op{N}$-smallest\footnote{Largest in $N$.} element $m(x) = \{ i \in I \mid x \in U_i \}$ for all $x \in \uall$, cellular since $\UU$ is good, and link-cofibrant since $\UU$ is tame.
    Thus $\uall$ realizes the homotopy type of $\op{N}$. This is the same as the homotopy type of~$N$, which is realized by $|\nerve(\UU)|$.
\end{proof}

    The finiteness requirement is inescapable when we are using closed sets as above. Consider the family $\UU = \{ \{t\} \mid t \in \Rr \}$ of singletons in~$\Rr$, contained in the tame class of finite subsets of~$\Rr$. Then $\uall = \Rr$ whereas $|\nerve(\UU)|$ is a discrete uncountable space.

With some bookkeeping we can confirm the naturality of the homotopy equivalence. The map in one direction can be factored through order complexes as follows:
\[
    \xymatrix{
    \uall
        \ar[r]
    &
    |\order(\op{N})|
        \ar@{=}[r]
    &
    |\order(N)|
        \ar[r]
    &
    |N|
    }    
\]
The first arrow is any $\op{N}$-map of CL-complexes, while the second arrow is any $N$-map of CL-complexes.

Here is a sample statement of naturality (with respect to a category of finite covers that has been left implicit but which the reader can easily infer).

\begin{proposition}[nerve theorem naturality]
\label{prop:nerve-natural}
    Let $\UU = (U_i \mid i \in I)$ and $\VV = (V_j \mid j \in J)$ be good tame covers in topological spaces $X,Y$ with $I, J$ finite. Let $f : X \to Y$ be a continuous map.
    Suppose $\UU$ is $f$-subordinate to $\VV$ in the sense that there exists a function $\alpha : I \to J$ such that $f(U_i) \subseteq V_{\alpha(i)}$ for all $i \in I$.
    Then the diagram
    \[
    \xymatrix{
    \uall
        \ar[rrr]_{\simeq}^{\mathbf{Theorem~\ref{thm:nerve}}}
        \ar[d]_-{f}
    &&&
    |N_\UU|
        \ar[d]^-{|\phi_\alpha|}
    \\
    \vall
        \ar[rrr]_{\simeq}^{\mathbf{Theorem~\ref{thm:nerve}}}
    &&&
    |N_\VV|
    }
    \]
    commutes in the homotopy category. Here $\phi_\alpha : N_{\UU} \to N_{\VV}$ is the poset map induced by~$\alpha$ (compare Example~\ref{ex:poset-system-maps}(2)).
\end{proposition}

\begin{remark}
    It follows from
    the statement of
    Proposition~\ref{prop:nerve-natural}
    that if $\beta: I \to J$ is another function such that $f(U_i) \subseteq V_{\beta(i)}$ for all $i \in I$, then $|\phi_\alpha|, |\phi_\beta| : |N_{\UU}| \to |N_{\VV}|$ are homotopic. What matters, then, is the existence of~$\alpha$ and not what specific $\alpha$ is chosen.
\end{remark}

\begin{proof}[Proof of Proposition~\ref{prop:nerve-natural}]
    We factor the diagram into three squares, where each of the eight objects is a CL-complex and each of the eight arrows is a CL-map.
    \[
    \xymatrix{
    \uall
        \ar[r]
        \ar[d]_-{f}
    &
    |\order(\op{N_{\UU}})|
        \ar@{=}[r]
        \ar[d]^-{|\order(\op{\phi_\alpha})|}
    &
    |\order(N_{\UU})|
        \ar[r]
        \ar[d]^-{|\order(\phi_\alpha)|}
    &
    |N_{\UU}|
        \ar[d]^-{|\phi_\alpha|}
    \\
    \vall
        \ar[r]
    &
    |\order(\op{N_{\VV}})|
        \ar@{=}[r]
    &
    |\order(N_{\VV})|
        \ar[r]
    &
    |N_{\VV}|
    }
    \]
    The four horizontal arrows are $\id$-maps (with respect to the posets $\op{N_{\UU}}$, $N_{\UU}$, $\op{N_{\VV}}$, $N_{\VV}$).
    The left two vertical arrows are $\op{\phi_{\alpha}}$-maps. The right two vertical arrows are $\phi_{\alpha}$-maps.
    Then the two paths around the left-hand square are both $\op{\phi_{\alpha}}$-maps so they are homotopic to each other. Likewise the two paths around the right-hand square are both $\phi_{\alpha}$-maps so they are homotopic to each other. Thus the two outer squares commute up to homotopy.
    The middle square commutes because $\order(\phi) = \order(\op{\phi})$ for any poset map~$\phi$.
\end{proof}

\subsection{Sheehy's multicover nerve theorem}
\label{subsec:sheehy}
I thank \'{A}ngel Javier Alonso for suggesting the following application of CL-complexes.
In~\cite{Sheehy_2012} Sheehy determined the homotopy type of the region in a topological space consisting of points belonging to at least $k$ sets in a cover~$\UU$, identifying it with the homotopy type of a particular subcomplex of the barycentric subdivision of the nerve of the cover. When $k=1$, the $k$-covered region is $\uall$ and the subcomplex is the full barycentric subdivision of the nerve, which is homeomorphic to the nerve itself so we recover the Nerve Theorem.

We generalize Sheehy's result slightly. Let $I$ be the indexing set for a cover~$\UU$ of a topological space. Let $K$ be an up-closed subset of the complete simplicial complex $\Delta(I)$. We may call this a \emph{covering requirement}. Then
\[
U^K
=
\bigcup_{\sigma \in K} U^\sigma
\]
is precisely the subset of the space consisting of points $x$ covered by `enough' of the sets $U_i$, in the sense that
$
\{ i \in I \mid x \in U_i \}
$
is a simplex in~$K$.

\begin{example}
\label{ex:Sheehy-k}
    If we define
    \[
    K_k
    =
    \{
    \sigma \in \Delta(I)
    \mid
    \text{$\sigma$ has at least $k$ vertices}
    \}
    \]
    then $U^{K_k}$ is the $k$-covered region.
\end{example}

We define the simplicial complex $\nerve(\UU)$ and the poset $N = N_{\UU}$ as in Equation~\ref{eq:nerve}.

\begin{theorem}[Sheehy's multicover nerve theorem]
\label{thm:multi-nerve}
    Let $\UU = (U_i \mid i \in I)$ be a good tame cover in a topological space, with $I$ finite. Let $K$ be a covering requirement.
    Then $U^K$ is the underlying space of a CL-complex for $\op{(K \cap N)}$. It is homotopy equivalent to the induced subcomplex of $\order(N)$ spanned by the vertices corresponding to the elements of~$K \cap N$.
\end{theorem}

Note that that $\order(N)$ can be identified with the barycentric subdivision of $\nerve(\UU)$.

\begin{proof}
Since $K$ is up-closed in~$\Delta(I)$, it follows that $\op{(K \cap N)}$ is closed in $\op{N}$ and therefore specifies a subcomplex of the CL-complex in Theorem~\ref{thm:nerve}. The underlying space of this subcomplex is
\[
\bigcup_{\sigma \in K \cap N} U^\sigma
\;=\;
\bigg[
\,
    \bigcup_{\sigma \in K \cap N} U^\sigma
\,
\bigg]
\,\cup\,
\bigg[
\,
    \bigcup_{\sigma \in K - N} \emptyset
\,
\bigg]
\;=\;
\bigcup_{\sigma \in K} U^\sigma
\;=\;
U^K
\]
which confirms the first assertion.

For the second assertion, we have a canonical homotopy equivalence
\[
\xymatrix{
U^K
    \ar[r]
&
|\order(\op{(K \cap N)})|
    \ar@{=}[r]
&
|\order(K \cap N)|
}
\]
where the first map is any continuous $\op{(K \cap N)}$-map of CL-complexes. Since chains in the poset $K \cap N$ are the same thing as chains in the poset~$N$ whose elements belong to $K \cap N$, it follows that $\order(K \cap N)$ is the specified induced subcomplex of $\order(N)$.
\end{proof}

We finish with a naturality statement.

\begin{proposition}[multicover nerve theorem naturality]
\label{prop:multi-nerve-natural}
    Let $\UU = (U_i \mid i \in I)$ and $\VV = (V_i \mid i \in I)$ be good tame covers in topological spaces $X,Y$ with $I$ finite. Let $f : X \to Y$ be a continuous map such that $f(U_i) \subseteq V_{i}$ for all $i \in I$.
    Then, for any covering requirement $K$, the diagram
    \[
    \xymatrix{
    U^K
        \ar[rrr]_-{\simeq}^-{\mathbf{Theorem~\ref{thm:multi-nerve}}}
        \ar[d]_-{f}
    &&&
    |\order(K \cap N_{\UU})|
        \ar[d]^-{\text{\rm inclusion}}
    \\
    V^K
        \ar[rrr]_-{\simeq}^-{\mathbf{Theorem~\ref{thm:multi-nerve}}}
    &&&
    |\order(K \cap N_{\VV})|
    }
    \]
    commutes in the homotopy category.
\end{proposition}

\begin{proof}
We factor the diagram into two squares.
\[
\xymatrix{
U^K
    \ar[r]
    \ar[d]_-{f}
&
|\order(\op{(K \cap N_{\UU})})|
    \ar@{=}[r]
    \ar[d]^-{\text{\rm inclusion}}
&
|\order(K \cap N_{\UU})|
    \ar[d]^-{\text{\rm inclusion}}
\\
V^K
    \ar[r]
&
|\order(\op{(K \cap N_{\VV})})|
    \ar@{=}[r]
&
|\order(K \cap N_{\VV})|
}
\]
This can be read as a `subdiagram' (restriction to subobjects) of the left two squares of the diagram in the proof of Proposition~\ref{prop:nerve-natural} (special case $I = J$, $\alpha = \id_I$) so it, too, commutes up to homotopy.
\end{proof}

\section*{Acknowledgements}

I gratefully thank Chad Giusti, Vladimir Itskov, Michael Robinson, Radmila Sazdanovic, Nikolas Schonsheck,
Melvin Vaupel, and Iris Yoon, my co-authors on~\cite{deSilva_GIRSSVY_2026}, for their generous encouragement to publish this work separately from our group project. My especial thanks to Iris for detailed comments on the entire manuscript, and to Niko and Melvin for helpful feedback on early material.
I am indebted to \'{A}ngel Javier Alonso for pointing out the connection to Don Sheehy's Multicover Nerve Theorem, and to Don himself for subsequent helpful conversations.
I thank Michael Kerber and Elena Wang for a momentum-generating invitation to give an early lecture on this work.

\bibliography{CL.bbl}


\end{document}